\documentclass[11pt,a4paper,reqno]{amsart}
\usepackage[left=30mm,top=40mm,right=30mm,bottom=35mm]{geometry}
\usepackage{amsfonts}
\usepackage{amsmath}
\usepackage{enumitem}
\usepackage{amssymb}
\usepackage{amsthm}
\usepackage{times}
\usepackage{color}
\usepackage[T1]{fontenc}
\usepackage{inputenc}
\usepackage{xcolor}
\usepackage{comment}
\usepackage{hyperref}
\hypersetup{
 colorlinks=true,
 linkcolor=blue,
 citecolor=blue,
 filecolor=blue,
 urlcolor=blue
}
\usepackage[nameinlink,capitalize,noabbrev]{cleveref}
\usepackage{aliascnt}
\usepackage{mathrsfs}
\usepackage{booktabs}
\usepackage{tikz}
\usepackage{stmaryrd}
\usepackage{float}

\newtheorem{theorem}{Theorem}[section]
\newtheorem*{theorem*}{Theorem}
\newtheorem*{maintheorem}{Main Theorem}

\newaliascnt{corollary}{theorem}
\newtheorem{corollary}[corollary]{Corollary}
\aliascntresetthe{corollary}
\crefname{corollary}{Corollary}{Corollaries}
\Crefname{corollary}{Corollary}{Corollaries}

\newaliascnt{lemma}{theorem}
\newtheorem{lemma}[lemma]{Lemma}
\aliascntresetthe{lemma}
\crefname{lemma}{Lemma}{Lemmas}
\Crefname{lemma}{Lemma}{Lemmas}

\newaliascnt{proposition}{theorem}
\newtheorem{proposition}[proposition]{Proposition}
\aliascntresetthe{proposition}
\crefname{proposition}{Proposition}{Propositions}
\Crefname{proposition}{Proposition}{Propositions}

\newaliascnt{problem}{theorem}

\aliascntresetthe{problem}
\crefname{problem}{Problem}{Problems}
\Crefname{problem}{Problem}{Problems}

\newaliascnt{conjecture}{theorem}

\aliascntresetthe{conjecture}
\crefname{conjecture}{Conjecture}{Conjectures}
\Crefname{conjecture}{Conjecture}{Conjectures}

\theoremstyle{remark}
\newaliascnt{remark}{theorem}
\newtheorem{remark}[remark]{Remark}
\aliascntresetthe{remark}
\crefname{remark}{Remark}{Remarks}
\Crefname{remark}{Remark}{Remarks}

\theoremstyle{definition}
\newaliascnt{definition}{theorem}
\newtheorem{definition}[definition]{Definition}
\aliascntresetthe{definition}
\crefname{definition}{Definition}{Definitions}
\Crefname{definition}{Definition}{Definitions}

\newaliascnt{notation}{theorem}

\aliascntresetthe{notation}
\crefname{notation}{Notation}{Notations}
\Crefname{notation}{Notation}{Notations}

\newaliascnt{example}{theorem}
\newtheorem{example}[example]{Example}
\aliascntresetthe{example}
\crefname{example}{Example}{Examples}
\Crefname{example}{Example}{Examples}

\theoremstyle{plain}

\def\Z{\mathbb{Z}}

\def\KK{\mathbf{k}}

\newcommand{\ZZ}[0]{\ensuremath{\mathbb{Z}}}
\newcommand{\GA}[0]{\ensuremath{\mathbb{G}_{\mathrm{a}}}}
\newcommand{\GM}[0]{\ensuremath{\mathbb{G}_{\mathrm{m}}}}
\newcommand{\AF}[0]{\ensuremath{\mathbb{A}}}

\newcommand{\RR}[0]{\ensuremath{\mathbb{R}}}

\newcommand{\LND}[0]{\ensuremath{\operatorname{LND}}}

\newcommand{\supp}[0]{\ensuremath{\operatorname{supp}}}
\newcommand{\Supp}[0]{\ensuremath{\operatorname{Supp}}}
\newcommand{\Aut}[0]{\ensuremath{\operatorname{Aut}}}

\newcommand{\cone}[0]{\ensuremath{\operatorname{cone}}}
\newcommand{\homo}[0]{\ensuremath{\operatorname{Hom}}}
\newcommand{\Der}[0]{\ensuremath{\operatorname{Der}}}

\newcommand{\Spec}{\operatorname{Spec}}
\newcommand{\GL}{\mathrm{GL}}

\newcommand {\red}[1]{{\color[rgb]{1,0,0} {#1}}}

\newif\ifreview
\reviewfalse

\newif\ifdecisions
\decisionstrue

\ifreview
  \newcommand{\rchange}[1]{{\color{red}#1}}
  \newcommand{\rdel}[1]{{\color{black!45}$\llbracket$\,\textsf{\footnotesize CUT:}\ #1\,$\rrbracket$}}
  \newcommand{\rcomment}[1]{%
    \par\smallskip
    \begingroup
      \footnotesize\color{purple!65!black}%
      \leftskip=2em \rightskip=2em \parindent=0pt
      \textsf{[R.\,D\'iaz]}\ \textit{#1}\par
    \endgroup
    \smallskip}
\else
  \newcommand{\rchange}[1]{#1}
  \newcommand{\rdel}[1]{}
  \newcommand{\rcomment}[1]{}
\fi

\ifdecisions
  \newcommand{\decision}[1]{%
    \par\medskip\noindent
    \fbox{\parbox{0.95\linewidth}{\red{#1}}}%
    \par\medskip}
\else
  \newcommand{\decision}[1]{}
\fi

\begin{document}

\title[Characterization of Stanley-Reisner varieties]{Characterization of Stanley-Reisner varieties by their automorphism group}

\author{Roberto D\'iaz}
\address{Departamento de Matem\'aticas, Facultad de Ciencias, Universidad de La Serena, Juan
Cisternas 1200, La Serena, Chile.}%
\email{roberto.diazv1@userena.cl}

\author{Jos\'e Alejandro Samper}
\address{Departamento de Matem\'aticas, Facultad de Matem\'aticas, Pontificia Universidad Cat\'olica de Chile, Santiago, Chile.}
\email{jsamper@uc.cl}

\date{\today}

\thanks{
 \mbox{\hspace{11pt}}{\it Key words}: Stanley-Reisner rings, locally nilpotent derivations, Demazure roots, automorphism groups, ind-groups.\\
 \mbox{\hspace{11pt}}The first author is supported by ANID/Fondecyt Iniciaci\'on No.~11260788.}

\begin{abstract}
\rchange{We study the automorphism ind-group of a Stanley-Reisner variety $X_\Delta$ through a
combinatorial toolkit on the underlying complex: a facet closure operator, its Demazure roots, and
the resulting dichotomy between \emph{exposed} and \emph{hidden} facets. Our main theorem is that a
\emph{fully exposed} complex, that is, one in which every facet has a private vertex, is recovered
from the ind-group: $\Aut(X_\Delta)\cong\Aut(X_{\Delta'})$ forces $\Delta\cong\Delta'$. The
hypothesis cannot be dropped, but it holds after one stabilization, so for arbitrary
$\Delta,\Delta'$ an isomorphism $\Aut(X_\Delta\times\AF^1)\cong\Aut(X_{\Delta'}\times\AF^1)$
already forces $\Delta\cong\Delta'$. At the opposite extreme, a fully hidden $\Delta$ gives
$\Aut(X_\Delta)=T_0\rtimes S(\Delta)$, never isomorphic to the ind-group of a non-rigid
Stanley-Reisner variety. The criterion decides graphs and skeleta, and every complex is homotopy
equivalent to a fully exposed one.}
\end{abstract}

\maketitle

\section*{Introduction}

Characterizing a geometric object by its automorphism group has a long history: in the differentiable category, Filipkiewicz \cite{Fil82} proved that any abstract isomorphism between the diffeomorphism groups of two smooth manifolds is induced by a diffeomorphism. In affine algebraic geometry the analogous problem asks to what extent a variety $X$ is determined by $\Aut(X)$. For the affine space $\AF^n$ it culminates in \cite{Kra17,KRvS21,CRX23}; beyond it, characterizations are known for affine toric varieties \cite{DLR26}, smooth affine spherical varieties \cite{RvS21,RvS25}, and affine surfaces with a torus action \cite{LRU23}, and the automorphism group of a rigid irreducible affine variety was described in \cite[Theorem~1]{AG17}; see also \cite{RUvS25,CERUvS25} for group-theoretical characterizations of rationality. All of these results concern \emph{irreducible} varieties. At the opposite end of the spectrum, the automorphism groups of zero-dimensional monomial algebras were computed in \cite{DLMFR24}, and finite-dimensional monomial algebras were shown in \cite{DLA24} to be determined by the identity component of their automorphism group among finite-dimensional local algebras with cotangent space of fixed dimension.

This paper initiates the analogous study for \emph{Stanley-Reisner varieties}, a combinatorially rich class of typically reducible affine schemes. Given a simplicial complex $\Delta$ on $[n]:=\{1,\dots,n\}$, its Stanley-Reisner ring is
$$
\KK[\Delta] := \KK[x_1,\dots,x_n]/I_\Delta,
\qquad
I_\Delta := \bigl(\mathbf{x}^\sigma \,:\, \sigma\notin\Delta\bigr),
\quad \mathbf{x}^\sigma:=\textstyle\prod_{i\in\sigma}x_i,
$$
and $X_\Delta:=\Spec\KK[\Delta]$ is the corresponding Stanley-Reisner variety: the union of the coordinate subspaces indexed by the facets of $\Delta$. Stanley-Reisner rings are a cornerstone of combinatorial commutative algebra \cite{MS05,HH11}; their modules of derivations were computed by Brumatti--Simis \cite{BS95} (see also \cite{Tad09}), and M\"uller \cite[Theorem~1]{Mul95} studied $\Der(\KK[\Delta])$ as an abstract Lie algebra, proving that complexes in a certain class are determined by it; this is the closest precursor of the present work, and \cref{rem:muller-comparison} compares the two theorems in full. Our guiding analogy is twofold: on the toric side, $\Aut(X)$ of an affine toric variety is governed by Demazure roots \rchange{\cite[\S\,4, n\textsuperscript{o}\,2, D\'efinition~1]{Dem70}}, \cite[Theorem~2.7]{Lie10}, \cite{AR15}, and in the setting of \cite{DLR26} the varieties determined by their automorphism group are the products $\AF^1\times Z$; on the monoid side, an affine monoid algebra remembers its combinatorial genesis \cite[Chapter~5]{BG09}, and the Stanley-Reisner counterpart is a theorem of Bruns--Gubeladze \cite{BG96}: an isomorphism of $\KK$-algebras $\KK[\Delta]\cong\KK[\Delta']$ forces $\Delta\cong\Delta'$. Stanley-Reisner varieties sit between these worlds: each irreducible component is an affine space, but the components are glued along coordinate subspaces, and it is this gluing that the automorphism group must recover. For non-normal toric varieties see \cite{DL24,GPT14}; for automorphisms of polyhedral algebras and arrangements of toric varieties, a graded setting containing Stanley-Reisner rings, see \cite[Theorem~5.2]{BG02}, to which the present paper is complementary in working with the full automorphism ind-group rather than its graded part.
\rcomment{pondr\'ia ``.'' final en ``the present work.'' y eliminar\'ia ``and \cref{rem:muller-comparison} compares the two theorems in full''.}

Throughout, $\KK$ denotes an algebraically closed field of characteristic zero and $I_\Delta$ is assumed to have no linear generators, i.e., every singleton $\{i\}$ is a face of $\Delta$; this costs nothing, since deleting the redundant variables makes any Stanley-Reisner ring isomorphic to one of this form. Finally, $\Delta$ is \emph{fully exposed} if every facet of $\Delta$ is exposed (\cref{def:rigid-facet}), i.e.\ has a private vertex; this is strictly stronger than $X_\Delta$ admitting a nontrivial $\GA$-action, and it holds automatically whenever $\Delta$ has a cone vertex (\cref{lem:cone-roots}). The word \emph{rigid} carries several established and mutually incompatible meanings in this context; a footnote in \cref{subsec:der-roots} records which one we adopt, and how the terminology used here relates to that of the earlier literature.

\begin{maintheorem}[\cref{thm:main}]
Let $\Delta$ and $\Delta'$ be fully exposed simplicial complexes on finite vertex sets. If there is an isomorphism of ind-groups $\Aut(X_\Delta)\cong\Aut(X_{\Delta'})$, then $\Delta\cong\Delta'$ as abstract simplicial complexes, and hence $X_\Delta\cong X_{\Delta'}$ as algebraic varieties.
\end{maintheorem}

\rchange{This is the first characterization of \emph{reducible} affine varieties by their full automorphism ind-group; the precedents are M\"uller's Lie-algebra theorem (\cref{rem:muller-comparison}) and, for graded automorphisms of a larger class, \cite[Theorem~5.2]{BG02}.}%
\rcomment{el parrafo dice ``parece ser la primera''... con toda confianza digo que ``esta es la primera''... [y sobre la cl\'ausula que sigue:] eso lo eliminar\'ia, ya se habl\'o de eso y ahora damos el enfoque desde la geometr\'ia, que es primera vez un resultado de este tipo considerando reducible.}
A complex with a cone vertex is automatically fully exposed (\cref{lem:cone-roots}), so every product $\AF^m\times X_\Gamma$ with $m\geq1$, in particular every affine space, falls within its scope. The sole hypothesis is sharp: \cref{prop:same-aut} exhibits two non-isomorphic Stanley-Reisner varieties whose automorphism ind-groups are isomorphic.

Sharp, but satisfied after a single stabilization: $X_\Delta\times\AF^1\cong X_{\Delta\ast v}$ is a cone, hence fully exposed, and coning is injective on isomorphism classes of complexes (\cref{lem:cone-cancel}). So for $\Delta,\Delta'$ subject to no hypothesis at all, an isomorphism of ind-groups $\Aut(X_\Delta\times\AF^1)\cong\Aut(X_{\Delta'}\times\AF^1)$ already forces $\Delta\cong\Delta'$ and $X_\Delta\cong X_{\Delta'}$ (\cref{cor:stabilized}). Cancellation itself holds in this class for a cheaper reason, with no automorphism groups involved: $X_\Delta\times\AF^1\cong X_{\Delta'}\times\AF^1$ gives $\KK[\Delta][x]\cong\KK[\Delta'][x']$, hence $\Delta\ast v\cong\Delta'\ast v'$ by \cite{BG96} and $\Delta\cong\Delta'$ (\cref{cor:sr-cancellation}). \cref{cor:stabilized} is strictly stronger than that: an isomorphism of the varieties $X_\Delta\times\AF^1\cong X_{\Delta'}\times\AF^1$ implies its hypothesis, and is a far more restrictive assumption.

\cref{sec:roots} is self-contained and can be read on its own. \rchange{It is built on a closure operator:} for a face $S$ of $\Delta$, let $\operatorname{cl}(S)$ be the intersection of the facets containing $S$. Each pair $(S,i)$ with $S\in\Delta$ and $i\in\operatorname{cl}(S)\setminus S$ contributes not one vector but a whole family: the \emph{roots} of $\Delta$ are the $\beta-\mathbf{e}_i\in\ZZ^n$ with $\beta\in\ZZ^n_{\geq0}$, $\beta_i=0$ and $i\in\operatorname{cl}(\supp\beta)\setminus\supp\beta$ (\cref{def:roots-comb}), one for each exponent vector $\beta$ with the given support, so that $\mathcal{R}(\Delta)\subseteq\ZZ^n$ is infinite as soon as it is nonempty. The non-squarefree roots are not redundant: passing to the squarefree part of $\mathcal{R}(\Delta)$ already breaks the reconstruction (\cref{ex:full-needed}). These vectors are exactly the Demazure roots of $X_\Delta$, that is, the degrees of the homogeneous locally nilpotent derivations of $\KK[\Delta]$ (\cref{thm:roots-agree}); so the $\GA$-actions on $X_\Delta$ \rchange{detect when $\operatorname{cl}$ is not the identity.} A facet with at least two vertices is \emph{exposed} when one of its vertices is \emph{private}\rchange{, i.e.\ its removal leaves a face lying in no other facet (equivalently, the facet admits an elementary collapse, \cref{rem:morse}),} and \emph{hidden} otherwise (\cref{def:rigid-facet,lem:private-vertex}); $\Delta$ is \emph{fully exposed} if every facet is exposed and \emph{fully hidden} if none is. In the standard families the dichotomy is decided by familiar invariants: a graph is fully exposed exactly when it is a disjoint union of stars without isolated vertices, and fully hidden exactly when it has no vertex of degree one; proper skeleta of simplices are fully hidden (\cref{prop:graph-exposed}). Fully exposed complexes are not generic\rchange{ (among graphs they are a vanishing fraction), but neither are they rare:} every cone is one, and every complex is homotopy equivalent to one, obtained by whiskering (\cref{prop:whisker-homotopy}).

\rchange{In this language, the following purely combinatorial result is central to the proof of the Main Theorem.}%
\rcomment{tenemos dos ``main theorem'' por lo que no queda claro, siento yo. Por lo que pas\'e el segundo a solo ``theorem''.}

\begin{theorem*}[\cref{thm:core} and \cref{cor:core-permuted}]
Let $\Delta$ and $\Delta'$ be fully exposed simplicial complexes on $[n]$. If $\sigma\cdot\mathcal{R}(\Delta)=\mathcal{R}(\Delta')$ for some permutation $\sigma\in S_n$ acting on $\ZZ^n$ by permuting coordinates, then $\sigma\cdot\Delta=\Delta'$; in particular $\Delta\cong\Delta'$.
\end{theorem*}

So on \rchange{fully exposed} complexes the root set is a complete
invariant; reconstruction breaks down as soon as one steps outside that class, already for the
pair of complexes behind the sharpness statement above (\cref{ex:rigid-invisible}). \rchange{The
geometric and combinatorial viewpoints are related as follows.}%
\rcomment{\textquestiondown por qu\'e no decir simplemente ``fully exposed'' aqu\'i?}

\medskip
\begin{center}
\begin{tabular}{@{}p{0.47\textwidth}p{0.47\textwidth}@{}}
\toprule
\textbf{Geometry of $X_\Delta$} & \textbf{Combinatorics of $\Delta$}\\
\midrule
Irreducible component $V_F\cong\AF^{|F|}$ & Facet, i.e.\ maximal face, $F\in\mathcal{F}(\Delta)$ (\cref{subsec:sr})\\[6pt]
Direct factor: $X_\Delta\cong Z\times\AF^1$ & Cone vertex: $\Delta=\Gamma\ast v$ \eqref{eq:cone-ring}, \cite{BG96}\\[6pt]
Maximal-dimensional torus in $\Aut(X_\Delta)$, unique up to conjugacy & The $n$ vertices; character lattice $\ZZ^{[n]}$ (\cref{thm:tori})\\[6pt]
Root subgroup $U_\alpha\cong\GA$ & Root $\alpha\in\mathcal{R}(\Delta)$ (\cref{prop:root-subgroups})\\[6pt]
Demazure root $\beta-\mathbf{e}_i$ of $X_\Delta$ & $\supp\beta\in\Delta$, $i\in\operatorname{cl}(\supp\beta)\setminus\supp\beta$ (\cref{thm:roots-agree})\\[6pt]
$X_\Delta$ is rigid (no $\GA$-action) & $\Delta$ is fully hidden: $\operatorname{cl}=\mathrm{id}$ (\cref{cor:hidden-rigid})\\[6pt]
$\Aut(X_\Delta)\cong T_0\rtimes S(\Delta)$, torus by symmetries & $\Delta$ is fully hidden (\cref{thm:rigid-aut,rem:rigid-vs-nonrigid})\\[6pt]
$\Aut(X_\Delta)$ determines $X_\Delta$ in the fully exposed class & $\Delta$ is fully exposed (\cref{thm:main}); one direction only\\[6pt]
$\Aut(X_\Delta\times\AF^1)$ determines $X_\Delta$ & No hypothesis on $\Delta$ (\cref{cor:stabilized})\\
\bottomrule
\end{tabular}
\end{center}
\medskip

\noindent The rows are equivalences down to the seventh. The eighth is an implication: we
prove that $\Aut(X_\Delta)$ determines $X_\Delta$ inside the fully exposed class, not that the
fully exposed complexes are the only ones so determined; and the last row is not a
correspondence but the statement that one stabilization removes the hypothesis altogether. The $T_0\rtimes S(\Delta)$ row does have
its converse (\cref{rem:rigid-vs-nonrigid}): a $\Delta$ that is not fully hidden has
$\mathcal{R}(\Delta)\neq\emptyset$, hence a root subgroup $\cong\GA$ inside $\Aut(X_\Delta)$, which
$T_0\rtimes S(\Delta)$ does not contain.
\medskip

The paper is organized as follows. After the background of \cref{sec:prelim}, \cref{sec:roots} is purely combinatorial: the closure
operator and $\mathcal{R}(\Delta)$, the exposed/hidden dichotomy, the reconstruction theorem, and
what the dichotomy amounts to for graphs, skeleta and whiskered complexes. \cref{sec:aut} crosses
to geometry: \cref{subsec:der-roots} identifies the degrees of the homogeneous locally nilpotent
derivations of $\KK[\Delta]$ with the combinatorial roots, and \cref{subsec:autgroup} builds the
ind-group structure of $\Aut(X_\Delta)$. \cref{sec:main} assembles these into the Main Theorem, compares it with
M\"uller's theorem (\cref{rem:muller-comparison}), and proves the stabilized statement
(\cref{cor:stabilized}).\rdel{ A companion note, available from the authors, collects three lines of
continuation: a conjectural description of $\Aut(X_\Delta)$ in terms of
$(\mathcal{R}(\Delta),S(\Delta))$, the interaction with whiskering, and the fibers of
$\Delta\mapsto\mathcal{R}(\Delta)$.}%

\subsection*{Acknowledgments}

We thank Andriy Regeta for his valuable guidance. This collaboration began at the AGREGA workshop held at Universidad de La Serena in January 2025, where the first author presented the problem; we thank the institution for its support and hospitality. The final stages of this work were carried out while the second author was visiting the \'Ecole de technologie sup\'erieure in Montreal; the second author thanks the institution for its hospitality.

The authors acknowledge the use of Claude (Anthropic) during the development of this work. This tool was used in the exploratory phase of the research, including discussions of formulations and the construction of examples. All mathematical arguments, proofs, and final verifications were carried out independently by the authors, who assume full responsibility for the contents of this paper.

\section{Preliminary}\label{sec:prelim}
\subsection{Simplicial complexes and Stanley-Reisner varieties}\label{subsec:sr}
Throughout, $\KK$ is an algebraically closed field of characteristic zero. A \emph{simplicial complex} $\Delta$ on $[n]:=\{1,\dots,n\}$ is a collection of subsets of $[n]$, called \emph{faces}, closed under taking subsets. The inclusion-maximal faces are the \emph{facets}, and $\mathcal{F}(\Delta)$ denotes the set of facets. For an arbitrary subset $\sigma\subseteq[n]$, face or not, we write $\mathcal{F}_\sigma(\Delta):=\{F\in\mathcal{F}(\Delta): \sigma\subseteq F\}$, abbreviating $\mathcal{F}_{\{i\}}(\Delta)=\mathcal{F}_i(\Delta)$; thus $\mathcal{F}_\emptyset(\Delta)=\mathcal{F}(\Delta)$, and $\mathcal{F}_\sigma(\Delta)=\emptyset$ exactly when $\sigma\notin\Delta$. Also, if $\KK[\mathbf{x}]:=\KK[x_1,\dots,x_n]$ and $\sigma\subseteq[n]$, we set $\mathbf{x}^{\sigma}=\prod_{i\in \sigma}x_i\in\KK[\mathbf{x}]$. We assume throughout that $\{i\}\in\Delta$ for every $i\in[n]$; equivalently, the Stanley-Reisner ideal
$$
I_\Delta := \bigl(\mathbf{x}^\sigma : \sigma\subseteq[n],\ \sigma\notin\Delta\bigr)\subseteq \KK[\mathbf{x}]
$$
has no linear generators, where $\operatorname{G}(I_\Delta)$ denotes its unique minimal monomial generating set (see \cite[Proposition~1.1.6]{HH11}); the elements of $\operatorname{G}(I_\Delta)$ are the monomials $\mathbf{x}^\sigma$ with $\sigma$ a minimal non-face. The \emph{Stanley-Reisner ring} and \emph{variety} are $\KK[\Delta]:=\KK[\mathbf{x}]/I_\Delta$ and $X_\Delta:=\Spec\KK[\Delta]$. The irredundant primary decomposition (see \cite[Lemma~1.5.4]{HH11})
$$
I_\Delta=\bigcap_{F\in\mathcal{F}(\Delta)}\mathfrak{p}_{F},\qquad \mathfrak{p}_{F}:=(x_l : l\notin F),
$$
shows that the irreducible components of $X_\Delta$ are the coordinate subspaces $V_F:=\{x_l=0 \text{ for } l\notin F\}\cong\AF^{|F|}$, $F\in\mathcal{F}(\Delta)$. We write $\mathfrak{q}_F:=\mathfrak{p}_F/I_\Delta\subseteq\KK[\Delta]$ for the corresponding minimal primes and $\KK[F]:=\KK[\Delta]/\mathfrak{q}_F\cong\KK[x_j: j\in F]$ for the coordinate rings of the components. More generally, for any nonempty family $\mathcal{S}\subseteq\mathcal{F}(\Delta)$ with $H:=\bigcap_{F\in\mathcal{S}}F$ one has
\begin{equation}\label{eq:intersections}
\sum_{F\in\mathcal{S}}\mathfrak{q}_F=(\overline{x}_l : l\notin H)
\qquad\text{and}\qquad
\KK[\Delta]\Big/\sum_{F\in\mathcal{S}}\mathfrak{q}_F\;\cong\;\KK[x_j : j\in H],
\end{equation}
a polynomial ring of dimension $|H|$: indeed, the kernel of $\KK[\mathbf{x}]\twoheadrightarrow \KK[\Delta]/(\overline{x}_l:l\notin H)$ equals $I_\Delta+(x_l:l\notin H)$, and every minimal generator $\mathbf{x}^\sigma$ of $I_\Delta$ satisfies $\sigma\not\subseteq H$ (as $H$ is a face), hence lies in $(x_l:l\notin H)$.
For $a\in\ZZ^n$ we write $\supp(a):=\{i\in[n]: a_i\neq0\}$ and $|a|:=\sum_{i}a_i$. The monomials $\overline{\mathbf{x}}^{a}$ with $a\in\ZZ^n_{\geq0}$ and $\supp(a)\in\Delta$ form a $\KK$-basis of $\KK[\Delta]$ (see \cite[Proposition~1.5.1]{HH11}). For $S\subseteq[n]$ we write $\mathbf{e}_S:=\sum_{l\in S}\mathbf{e}_l\in\ZZ^n$; in particular $\mathbf{e}_{\{i\}}=\mathbf{e}_i$ is the $i$-th standard basis vector. In particular $\KK[\Delta]$ is $\ZZ^n$-graded, $\KK[\Delta]=\bigoplus_{a\in\ZZ^n}\overline{B}_a$ with $\overline{B}_a=\KK\cdot\overline{\mathbf{x}}^a$ when $a\in\ZZ^n_{\geq0}$ and $\supp(a)\in\Delta$ and $\overline{B}_a=0$ otherwise, so that all graded pieces have dimension at most one; coarsening by total degree gives the standard $\ZZ$-grading $\KK[\Delta]=\bigoplus_{e\geq0}A_e$ with $A_e:=\bigoplus_{|a|=e}\overline{B}_a$. Moreover the product of two monomials is $\overline{\mathbf{x}}^{a}\,\overline{\mathbf{x}}^{b}=\overline{\mathbf{x}}^{a+b}$, which vanishes if and only if $\supp(a+b)=\supp(a)\cup\supp(b)\notin\Delta$.
A vertex $i\in[n]$ is a \emph{cone vertex} of $\Delta$ if $i\in F$ for every facet $F$; we write $c(\Delta)\subseteq[n]$ for the set of cone vertices of $\Delta$. For a new vertex $v\notin[n]$, the \emph{cone} $\Delta\ast v:=\Delta\cup\{S\cup\{v\}: S\in\Delta\}$ is a simplicial complex on $[n]\cup\{v\}$ with the same minimal non-faces as $\Delta$, since a subset of $[n]\cup\{v\}$ is a non-face of $\Delta\ast v$ exactly when its intersection with $[n]$ is a non-face of $\Delta$; hence
\begin{equation}\label{eq:cone-ring}
I_{\Delta\ast v}=I_\Delta\cdot\KK[\mathbf{x},x_v],
\qquad
\KK[\Delta\ast v]=\KK[\Delta][x_v],
\qquad
X_{\Delta\ast v}\cong X_\Delta\times\AF^1,
\end{equation}
and $I_{\Delta\ast v}$ again has no linear generators, so $\Delta\ast v$ satisfies our standing convention. We write $S(\Delta)\subseteq S_n$ for the group of permutations $\sigma$ of $[n]$ with $\sigma\cdot\Delta:=\{\sigma(A):A\in\Delta\}=\Delta$; each $\sigma\in S(\Delta)$ induces the automorphism $\rho_\sigma$ of $\KK[\Delta]$, $\overline{x}_i\mapsto\overline{x}_{\sigma(i)}$.

\subsection{Derivations}\label{subsec:der}

A \emph{derivation} of a $\KK$-algebra $A$ is a $\KK$-linear map $\partial\colon A\to A$ satisfying the Leibniz rule $\partial(fg)=\partial(f)g+f\partial(g)$ for all $f,g\in A$. A derivation $\partial$ is \emph{locally nilpotent} if for every $f\in A$ there exists $k\geq1$ with $\partial^k(f)=0$; we write $\LND(A)$ for the set of locally nilpotent derivations. A derivation $\partial$ of $\KK[\Delta]$ is \emph{$\ZZ^n$-homogeneous of degree $\alpha\in\ZZ^n$} if $\partial(\overline{B}_a)\subseteq\overline{B}_{a+\alpha}$ for every $a\in\ZZ^n$.

Every derivation $\partial$ of $\KK[\Delta]$ decomposes uniquely as a finite sum $\partial=\sum_{\alpha\in\ZZ^n}\partial_\alpha$ of $\ZZ^n$-homogeneous derivations, where $\partial_\alpha(f):=(\partial f)_{a+\alpha}$ for $f$ homogeneous of degree $a$; the sum is finite because $\partial_\alpha$ vanishes on the generators $\overline{x}_1,\dots,\overline{x}_n$ for all but finitely many $\alpha$ \cite[Proposition~2.1]{DL24}. In particular, to verify whether $\KK[\Delta]$ admits a nonzero locally nilpotent derivation it suffices to look among the $\ZZ^n$-homogeneous ones.

\begin{lemma}\label{lem:decomposition}
Let $\partial\neq0$ be a locally nilpotent derivation of $\KK[\Delta]$, with homogeneous decomposition $\partial=\sum_\alpha\partial_\alpha$. Then at least one nonzero component $\partial_\alpha$ is locally nilpotent.
\end{lemma}
\begin{proof}
Let $A=\{\alpha\in\ZZ^n:\partial_\alpha\neq0\}$, which is finite. Since $A$ is finite, there exists a linear form $\varphi:\ZZ^n\to\ZZ$ that is injective on $A$: indeed, it suffices to choose $\varphi$ so that it does not vanish on any nonzero difference $\alpha-\beta$ with $\alpha,\beta\in A$, $\alpha\neq\beta$. Let $\tau\in A$ be the unique element maximizing $\varphi$ on $A$.

We claim that $\partial_\tau$ is locally nilpotent. Suppose otherwise. Then there exists $f$ such that $\partial_\tau^k(f)\neq0$ for every $k\ge0$. Writing $f=\sum_a f_a$ into homogeneous components and noting that $\partial_\tau^k(f_a)$ has degree $a+k\tau$, these terms lie in pairwise distinct homogeneous components. Hence there exists a homogeneous component $g=f_a$ such that $\partial_\tau^k(g)\neq0$ for infinitely many $k$, and therefore for all $k\ge0$.

Expanding $\partial^k=\sum\partial_{\alpha_1}\cdots\partial_{\alpha_k}$ over words $(\alpha_1,\dots,\alpha_k)\in A^k$, a word contributes to the homogeneous component of degree $a+\alpha_1+\cdots+\alpha_k$. Since $\varphi(\alpha)\le\varphi(\tau)$ for every $\alpha\in A$, with equality only when $\alpha=\tau$, we have $\varphi(\alpha_1+\cdots+\alpha_k)<k\varphi(\tau)$ unless $\alpha_1=\cdots=\alpha_k=\tau$. Hence $(\partial^k(g))_{a+k\tau}=\partial_\tau^k(g)\neq0$ for every $k$, contradicting the local nilpotency of $\partial$.
\end{proof}

\subsection{Ind-varieties and ind-groups}\label{subsec:ind}
An \emph{ind-variety} is a set $\mathcal{V}$ with an ascending filtration
$\mathcal{V}_0\subseteq\mathcal{V}_1\subseteq\cdots$ with union $\mathcal{V}$, in which each
$\mathcal{V}_k$ is an algebraic variety (not necessarily irreducible) and each inclusion
$\mathcal{V}_k\hookrightarrow\mathcal{V}_{k+1}$ is a closed immersion; a \emph{morphism} of
ind-varieties $\varphi\colon\mathcal{V}\to\mathcal{W}$ is a map such that for every $k$ there is
an $\ell$ with $\varphi(\mathcal{V}_k)\subseteq\mathcal{W}_\ell$ and
$\varphi|_{\mathcal{V}_k}$ a morphism of varieties \cite[Definition~1.1.1]{FK18}. An
\emph{ind-group} is a group $\mathcal{G}$ with an ind-variety structure for which
multiplication and inversion are morphisms of ind-varieties \cite[Section~2.1]{FK18}, and an
\emph{algebraic subgroup} of $\mathcal{G}=\bigcup_k\mathcal{G}_k$ is a subgroup that is a closed
algebraic subvariety of some $\mathcal{G}_k$ (see \cite[Section~2.4]{FK18}). For an algebraic
group $T$ we write $X(T):=\homo(T,\GM)$ for its character group, and by a \emph{torus} in
$\mathcal{G}$ we mean an algebraic subgroup isomorphic to some $(\KK^*)^m$. An
\emph{isomorphism of ind-groups} is a group isomorphism that is an isomorphism of
ind-varieties.

For any finitely generated associative $\KK$-algebra $R$ the group $\Aut_\KK(R)$ carries a natural structure of affine ind-group \cite[Proposition~5.4.1]{FK18}; in particular so does $\Aut(X_\Delta)=\Aut_\KK(\KK[\Delta])$. Write $A:=\KK[\Delta]$ and put $A_{\leq d}:=\bigoplus_{e\leq d}A_e$, a finite-dimensional vector space. A $\KK$-algebra endomorphism of $A$ is determined by the images of the generators, and a tuple
$(f_1,\dots,f_n)$ arises this way precisely when $\prod_{i\in\sigma}f_i=0$ for every
$\mathbf{x}^\sigma\in\operatorname{G}(I_\Delta)$, which is a closed condition. This realises
$\operatorname{End}_\KK(A)$ as a closed subset of $\bigcup_d (A_{\leq d})^n$, and bounding the
degrees of the images of the generators stratifies it by the affine varieties
$\operatorname{End}_\KK(A)_{\leq d}$, whose inclusions are closed immersions. Then
$$
\Aut(X_\Delta)\;=\;\bigl\{(\varphi,\psi)\in\operatorname{End}_\KK(A)^2 \,:\, \varphi\circ\psi=\psi\circ\varphi=\operatorname{id}\bigr\}
$$
is closed in $\operatorname{End}_\KK(A)^2$, with $d$-th piece the pairs $(\varphi,\varphi^{-1})$
such that $\varphi$ and $\varphi^{-1}$ both send every $\overline{x}_l$ to an element of degree
at most $d$. Every algebraic subgroup $G\subseteq\Aut(X_\Delta)$ acts algebraically on $X_\Delta$: the action morphism $G\times X_\Delta\to X_\Delta$ is obtained by restricting the tautological action of the filtration piece containing $G$.

A \emph{root subgroup} of an ind-group $\mathcal{G}$ with respect to a torus
$T\subseteq\mathcal{G}$ is an algebraic subgroup $U\subseteq\mathcal{G}$ isomorphic to $\GA$ for
which there are an isomorphism $\varepsilon\colon\GA\to U$ and a nontrivial character
$\alpha\in X(T)$ with $t\,\varepsilon(s)\,t^{-1}=\varepsilon(\alpha(t)\cdot s)$ for all $t\in T$
and $s\in\GA$ (so $U$ is normalized by $T$; cf.~\cite[Definition~6.1]{Kra17}, where $\alpha$ may
be trivial). The character $\alpha$, independent of the choice of $\varepsilon$, is the
\emph{weight} of $U$.

\section{Demazure roots of a simplicial complex}\label{sec:roots}

This section is purely combinatorial. We attach to $\Delta$ a closure operator on its face
poset, define the Demazure roots $\mathcal{R}(\Delta)\subseteq\ZZ^n$ in terms of it, and show
that a complex all of whose facets are \emph{exposed} is determined by $\mathcal{R}(\Delta)$
(\cref{thm:core}). The name is justified in \cref{thm:roots-agree}; nothing below depends on
that.

\subsection{The facet closure operator}\label{subsec:closure}

For $S\in\Delta$ put $\operatorname{cl}_\Delta(S):=\bigcap_{F\in\mathcal{F}_S(\Delta)}F$, the
intersection being over a nonempty family. This is the closure operator of the
intersection-closed system generated by $\mathcal{F}(\Delta)$: it is extensive, monotone and
idempotent, its fixed points are the faces that are intersections of facets, and
$\operatorname{cl}(\emptyset)=c(\Delta)$, the set of cone vertices. We drop the subscript when
$\Delta$ is clear.

\begin{definition}\label{def:Ai}
For $i\in[n]$ set $\mathcal{A}_i(\Delta):=\{S\subseteq[n]\setminus\{i\}: S\in\Delta,\
i\in\operatorname{cl}(S)\}$.
\end{definition}

\begin{definition}\label{def:roots-comb}
A \emph{Demazure root} of $\Delta$ is a vector $\beta-\mathbf{e}_i\in\ZZ^n$ with
$\beta\in\ZZ^n_{\geq0}$, $\beta_i=0$ and $\supp(\beta)\in\mathcal{A}_i(\Delta)$; that is,
\begin{equation}\label{eq:R-closure}
\mathcal{R}(\Delta)=\bigl\{\beta-\mathbf{e}_i \;:\; \beta\in\ZZ^n_{\geq0},\ \supp(\beta)\in\Delta,\ i\in\operatorname{cl}(\supp\beta)\setminus\supp\beta\bigr\},
\end{equation}
with $\supp(0)=\emptyset$, so $\beta=0$ contributes the $-\mathbf{e}_i$ with $i\in c(\Delta)$.
We write $\mathcal{R}_i(\Delta):=\{\alpha\in\mathcal{R}(\Delta):\alpha_i=-1\}$.
\end{definition}

The equivalent facet-theoretic form below is what \cref{subsec:der-roots} consumes.

\begin{lemma}\label{lem:closure-Ai}
Let $i\in[n]$ and $S\subseteq[n]\setminus\{i\}$. Then $S\in\mathcal{A}_i(\Delta)$ if and only if
$\emptyset\neq\mathcal{F}_S(\Delta)\subseteq\mathcal{F}_i(\Delta)$; in particular
$\emptyset\in\mathcal{A}_i(\Delta)$ if and only if $i$ is a cone vertex.
\end{lemma}

\rchange{\begin{proof}
$S\in\mathcal{A}_i(\Delta)$ if and only if $S\in\Delta$ and
$i\in\operatorname{cl}(S)=\bigcap\mathcal{F}_S(\Delta)$, if and only if
$\mathcal{F}_S(\Delta)\neq\emptyset$ and every facet containing $S$ contains $i$, if and only
if $\emptyset\neq\mathcal{F}_S(\Delta)\subseteq\mathcal{F}_i(\Delta)$. Taking $S=\emptyset$
gives $\mathcal{F}_\emptyset(\Delta)=\mathcal{F}(\Delta)\neq\emptyset$, so the condition becomes
$\mathcal{F}(\Delta)=\mathcal{F}_i(\Delta)$, i.e.\ $i$ lies in every facet.
\end{proof}}

\begin{lemma}\label{lem:root-data}\label{lem:R-equivariant}
\textup{(i)} Every $\alpha\in\mathcal{R}(\Delta)$ has exactly one negative coordinate, equal to
$-1$; its index is the \emph{target} $t(\alpha)$, the \emph{support} is
$S(\alpha):=\supp(\alpha+\mathbf{e}_{t(\alpha)})\in\mathcal{A}_{t(\alpha)}(\Delta)$, the
representing pair in \cref{def:roots-comb} is unique, and $0\notin\mathcal{R}(\Delta)$.
\textup{(ii)} \emph{\textup{(}Equivariance.\textup{)}} For $\sigma\in S_n$,
$\operatorname{cl}_{\sigma\cdot\Delta}\circ\,\sigma=\sigma\circ\operatorname{cl}_\Delta$,
$\mathcal{A}_{\sigma(i)}(\sigma\cdot\Delta)=\sigma(\mathcal{A}_i(\Delta))$ and
$\mathcal{R}(\sigma\cdot\Delta)=\sigma\cdot\mathcal{R}(\Delta)$.
\end{lemma}

\rchange{\begin{proof}
\textup{(i)} Writing $\alpha=\beta-\mathbf{e}_i$ as in \cref{def:roots-comb}, we have
$\alpha_i=-1$ and $\alpha_l=\beta_l\geq0$ for $l\neq i$; so $i$ and $\beta=\alpha+\mathbf{e}_i$
are recovered from $\alpha$, the representing pair is unique, and $\alpha\neq0$.

\textup{(ii)} A facet $F$ of $\Delta$ contains $S$ if and only if $\sigma(F)$ is a facet of
$\sigma\cdot\Delta$ containing $\sigma(S)$, so
$\mathcal{F}_{\sigma(S)}(\sigma\cdot\Delta)=\sigma(\mathcal{F}_S(\Delta))$. Intersecting gives
$\operatorname{cl}_{\sigma\cdot\Delta}(\sigma(S))=\sigma(\operatorname{cl}_\Delta(S))$. Then
$\sigma(S)\in\mathcal{A}_{\sigma(i)}(\sigma\cdot\Delta)$ if and only if
$\sigma(i)\in\operatorname{cl}_{\sigma\cdot\Delta}(\sigma(S))=\sigma(\operatorname{cl}_\Delta(S))$,
if and only if $i\in\operatorname{cl}_\Delta(S)$, if and only if $S\in\mathcal{A}_i(\Delta)$.
Finally $\sigma\cdot(\beta-\mathbf{e}_i)=\sigma\cdot\beta-\mathbf{e}_{\sigma(i)}$ and
$\supp(\sigma\cdot\beta)=\sigma(\supp\beta)$, so
$\mathcal{R}(\sigma\cdot\Delta)=\sigma\cdot\mathcal{R}(\Delta)$. In particular
$\mathcal{R}(\Delta)$ and the family $(\mathcal{A}_i(\Delta))_{i\in[n]}$ determine each other.
\end{proof}}%
\rcomment{Lema antrior sugiero dejar solo la parte b. y agrgar prueba.}

\begin{remark}\label{rem:demazure-name}
The terminology is borrowed from the toric setting, where the homogeneous locally nilpotent
derivations of an affine toric variety correspond to the Demazure roots of its defining cone
\cite{Dem70,Lie10,DL24,DLMFR24}. \cref{thm:roots-agree} shows that \cref{def:roots-comb} is the
correct Stanley-Reisner analogue. Until then, $\mathcal{R}(\Delta)$ is a combinatorial
invariant of $\Delta$ and nothing more.
\end{remark}

\begin{remark}\label{rem:links}
Both $\operatorname{cl}$ and the $\mathcal{A}_i(\Delta)$ read compactly in terms of links, which
is the form recognised in \cref{rem:muller-dictionary}. The facets of
$\operatorname{lk}_\Delta(S):=\{\tau\in\Delta:\tau\cap S=\emptyset,\ \tau\cup S\in\Delta\}$ are the
$F\setminus S$ with $F\in\mathcal{F}_S(\Delta)$, so $i$ lies in every facet containing $S$ exactly
when $i$ is a cone vertex of the link (in particular $i\notin S$):
$$
\operatorname{cl}(S)=S\cup c\bigl(\operatorname{lk}_\Delta(S)\bigr),
\qquad
\mathcal{A}_i(\Delta)=\bigl\{S\in\Delta\ :\ i\in c(\operatorname{lk}_\Delta(S))\bigr\}.
$$
No facet occurs in any $\mathcal{A}_i(\Delta)$, since $\operatorname{lk}_\Delta(F)=\{\emptyset\}$.
Thus $\mathcal{R}(\Delta)$ is the ledger of which links of $\Delta$ are cones.
\end{remark}

\subsection{Exposed and hidden facets}\label{subsec:exposed}

\begin{definition}\label{def:rigid-facet}\label{def:rigidity-free}%
\label{def:fully-hidden}\label{def:private-vertex}
A facet $F$ is \emph{exposed} if there exist $i\in[n]$ and $S\in\mathcal{A}_i(\Delta)$ with
$S\subseteq F$, and \emph{hidden} otherwise. In that situation $i\in F$
automatically: by \cref{lem:closure-Ai} every facet containing $S$ contains $i$, and $F$ is one
of them. The complex $\Delta$ is
\emph{fully exposed} if every facet is exposed, and \emph{fully hidden} if no
facet is. A vertex $i\in F$ is \emph{private} for $F$ if
$\mathcal{F}_{F\setminus\{i\}}(\Delta)=\{F\}$, equivalently if
$F=\operatorname{cl}(F\setminus\{i\})$.
\end{definition}

\begin{lemma}\label{lem:private-vertex}
Let $F\in\mathcal{F}(\Delta)$ with $|F|\geq2$. Then $F$ is exposed if and only if $F$ has a
private vertex. If $n\geq2$, every facet with $|F|=1$ is hidden.
\end{lemma}

\begin{proof}
If $S\in\mathcal{A}_i(\Delta)$ with $S\subseteq F$ then $i\in F$, $S\subseteq F\setminus\{i\}$ and
$\mathcal{F}_{F\setminus\{i\}}(\Delta)\subseteq\mathcal{F}_S(\Delta)\subseteq\mathcal{F}_i(\Delta)$
by \cref{lem:closure-Ai}, so a facet containing $F\setminus\{i\}$ contains $i$, hence $F$, hence
equals $F$. Conversely a private $i$ gives $S:=F\setminus\{i\}\neq\emptyset$ with
$\mathcal{F}_S(\Delta)=\{F\}\subseteq\mathcal{F}_i(\Delta)$. If $F=\{i\}$ is exposed then
$S\subseteq F$, $i\notin S$ force $S=\emptyset$, so $i$ is a cone vertex and maximality gives
$\mathcal{F}(\Delta)=\{\{i\}\}$, i.e.\ $n=1$.
\end{proof}

\begin{corollary}\label{cor:fe-private}
For $n\geq2$, $\Delta$ is fully exposed if and only if every facet has a private vertex.\qed
\end{corollary}

\begin{example}\label{ex:exposed-hidden}
\cref{fig:exposed-hidden} shows the dichotomy on three small graphs (hidden facets are drawn thick).
In~(a) the middle edge $\{2,3\}$ is hidden and the two outer edges are exposed, with private vertices $2$ and $3$.
In~(b) the boundary of a triangle is fully hidden: no edge has a private vertex.
In~(c) the star is fully exposed: every edge has a private leaf.
For graphs the dichotomy is characterised in \cref{prop:graph-exposed}.
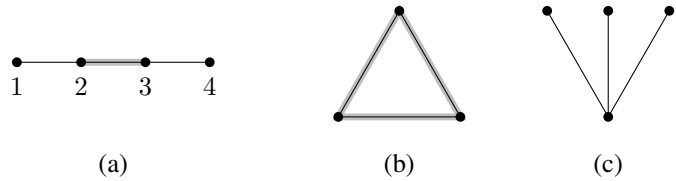
\begin{figure}[H]
\centering
\begin{tikzpicture}[scale=0.85, v/.style={circle,fill,inner sep=1.3pt},
                    hid/.style={line width=2.6pt,gray!50}]
\begin{scope}
  \draw[hid] (1,0.85)--(2,0.85);
  \draw (0,0.85)--(3,0.85);
  \foreach \x/\l in {0/1,1/2,2/3,3/4} {\node[v,label=below:{\small $\l$}] at (\x,0.85) {};}
  \node at (1.5,-0.75) {\small (a)};
\end{scope}
\begin{scope}[xshift=5cm]
  \draw[hid] (0,0)--(1.9,0)--(0.95,1.65)--cycle;
  \draw (0,0)--(1.9,0)--(0.95,1.65)--cycle;
  \node[v] at (0,0) {}; \node[v] at (1.9,0) {}; \node[v] at (0.95,1.65) {};
  \node at (0.95,-0.75) {\small (b)};
\end{scope}
\begin{scope}[xshift=9.2cm]
  \draw (0,0)--(-0.95,1.65); \draw (0,0)--(0.95,1.65); \draw (0,0)--(0,-0.05);
  \draw (0,0)--(0,0);
  \node[v] at (0,0) {};
  \node[v] at (-0.95,1.65) {}; \node[v] at (0.95,1.65) {}; \node[v] at (0,1.65) {};
  \draw (0,0)--(0,1.65);
  \node at (0,-0.75) {\small (c)};
\end{scope}
\end{tikzpicture}
\caption{The exposed/hidden dichotomy on three small graphs. Hidden facets are drawn thick.}
\label{fig:exposed-hidden}
\end{figure}
\end{example}

\begin{remark}\label{rem:morse}
In discrete Morse theory \cite{For98,Koz08} a face contained in exactly one facet, of dimension
one larger, is a \emph{free face}, and the pair defines an elementary collapse. Thus a facet
$F$ with $|F|\geq2$ is exposed precisely when some $F\setminus\{i\}$ is free, i.e.\ when $F$
admits an elementary collapse; a fully exposed complex is one in which every facet does so
individually, the collapses being required neither to be compatible nor to be iterated. Being
fully exposed is therefore incomparable with collapsibility: the path with facets
$\{1,2\},\{2,3\},\{3,4\}$ is collapsible but has the hidden facet $\{2,3\}$, while two disjoint
edges form a fully exposed complex that is not collapsible.
\end{remark}

\begin{proposition}\label{prop:rigid-cl}
The following are equivalent:
\begin{enumerate}[label=\textup{(\alph*)}]
\item\label{rc.a} $\mathcal{R}(\Delta)=\emptyset$;
\item\label{rc.b} $\operatorname{cl}=\mathrm{id}_\Delta$;
\item\label{rc.c} every face of $\Delta$ is an intersection of facets;
\item\label{rc.d} $\mathcal{A}_i(\Delta)=\emptyset$ for every $i\in[n]$;
\item\label{rc.e} $\Delta$ is fully hidden.
\end{enumerate}
\end{proposition}

\rchange{\begin{proof}
The equivalence of \ref{rc.b} and \ref{rc.c} is the definition of $\operatorname{cl}$ together
with $S\subseteq\operatorname{cl}(S)$ for every $S\in\Delta$. Conditions \ref{rc.b} and
\ref{rc.d} are equivalent because some $\mathcal{A}_i(\Delta)$ is nonempty if and only if
$\operatorname{cl}(S)\supsetneq S$ for some $S\in\Delta$. Conditions \ref{rc.d} and \ref{rc.a}
are equivalent because a root with target $i$ exists if and only if
$\mathcal{A}_i(\Delta)\neq\emptyset$, by \cref{def:roots-comb} and \cref{lem:root-data}.
Finally, \ref{rc.d} and \ref{rc.e} are equivalent: if $S\in\mathcal{A}_i(\Delta)$ then any
facet containing $S$ is exposed; conversely an exposed facet exhibits a nonempty
$\mathcal{A}_i(\Delta)$.
\end{proof}}%
\rcomment{Solo por estilo consistente sugiero la prueba anterior dejarla como sigue.}

\begin{remark}\label{rem:muller-dictionary}
Both extremes already have names in \cite{Mul95}. M\"uller calls $\sigma\in\Delta$
\emph{bounding} if it lies in $\operatorname{lk}_\Delta(\{j\})$ for exactly one vertex $j$,
writes $\partial\Delta$ for the subcomplex generated by the bounding faces, and calls $\Delta$
\emph{closed} if $\partial\Delta=\emptyset$ \cite[p.~135]{Mul95}. Since
$\sigma\in\operatorname{lk}_\Delta(\{j\})$ if and only if
$\{j\}\in\operatorname{lk}_\Delta(\sigma)$, the vertices $j$ with
$\sigma\in\operatorname{lk}_\Delta(\{j\})$ are the vertices of
$\operatorname{lk}_\Delta(\sigma)$; so $\sigma$ is bounding exactly when
$\operatorname{lk}_\Delta(\sigma)=\{\emptyset,\{i\}\}$ for a single vertex $i$, in which case
$F:=\sigma\cup\{i\}$ is a facet and $\sigma=F\setminus\{i\}$. By \cref{def:private-vertex} and
\cref{rem:links} this says that $i$ is a private vertex of $F$ if and only if $F\setminus\{i\}$
is a bounding face, and every bounding face arises so. Hence $\Delta$ is fully exposed exactly
when every facet has a bounding codimension-one face, and
$$
\Delta\ \text{is fully hidden}
\quad\Longleftrightarrow\quad\mathcal{R}(\Delta)=\emptyset
\quad\Longleftrightarrow\quad\partial\Delta=\emptyset,
$$
the first equivalence being \ref{rc.a}$\Leftrightarrow$\ref{rc.e} of \cref{prop:rigid-cl}. The
second can also be read off from \cite[Proposition~4(iii)]{Mul95}, because the \emph{negative}
roots of \cite[Proposition~3]{Mul95} are exactly the elements of $\mathcal{R}(\Delta)$: that
proposition declares $\beta-\mathbf{e}_i$, with $\beta\in\ZZ^n_{\geq0}$ and $\beta_i=0$, to be a
root exactly when $S:=\supp\beta$ is a face with $i\notin S$ lying in no face $\sigma$ with
$\sigma\cup\{i\}\notin\Delta$, i.e.\ exactly when $S\in\mathcal{A}_i(\Delta)$.
\end{remark}

A cone vertex is visible in $\mathcal{R}(\Delta)$ in the strongest way: the roots targeting it
record \emph{all} faces of $\Delta$.

\begin{lemma}\label{lem:cone-roots}
For $i\in[n]$ the following are equivalent:
\begin{enumerate}[label=\textup{(\alph*)}]
\item\label{cr.a} $i$ is a cone vertex of $\Delta$;
\item\label{cr.b} $\emptyset\in\mathcal{A}_i(\Delta)$;
\item\label{cr.c} $-\mathbf{e}_i\in\mathcal{R}(\Delta)$.
\end{enumerate}
In that case
\begin{enumerate}[label=\textup{(\arabic*)}]
\item\label{cr.1} $\mathcal{A}_i(\Delta)=\{S\in\Delta: i\notin S\}$;
\item\label{cr.2} $\Delta=\mathcal{A}_i(\Delta)\cup\{S\cup\{i\}:S\in\mathcal{A}_i(\Delta)\}$, so $\mathcal{R}(\Delta)$ determines $\Delta$;
\item\label{cr.3} $\Delta$ is fully exposed.
\end{enumerate}
\end{lemma}

\begin{proof}
\ref{cr.a}$\Leftrightarrow$\ref{cr.b} is \cref{lem:closure-Ai};
\ref{cr.b}$\Leftrightarrow$\ref{cr.c} since $-\mathbf{e}_i=\beta-\mathbf{e}_i$ forces $\beta=0$.
If $i$ is a cone vertex then $\mathcal{F}_S(\Delta)\subseteq\mathcal{F}(\Delta)=\mathcal{F}_i(\Delta)$
always, so $S\in\mathcal{A}_i(\Delta)$ reduces to $S\in\Delta$, which is \ref{cr.1}; \ref{cr.2}
follows since $T\in\Delta$ is $T$ or $(T\setminus\{i\})\cup\{i\}$ according as $i\notin T$ or
$i\in T$, and $S\in\mathcal{A}_i(\Delta)$ gives $S\cup\{i\}\subseteq F$ for any facet
$F\supseteq S$; and \ref{cr.3} because $\emptyset\subseteq F$ for every facet.
\end{proof}

Coning can therefore be undone.

\begin{lemma}\label{lem:cone-cancel}
If $w$ is a cone vertex of $\Lambda$, then $\Lambda=\operatorname{lk}_\Lambda(w)\ast w$.
Consequently coning is injective on isomorphism classes: for new vertices $v,v'$, if
$\Delta\ast v\cong\Delta'\ast v'$ then $\Delta\cong\Delta'$.
\end{lemma}

\begin{proof}
Since $w$ lies in every facet, $S\cup\{w\}\in\Lambda$ for every $S\in\Lambda$; hence
$\operatorname{lk}_\Lambda(w)=\{S\in\Lambda:w\notin S\}$ and $\Lambda=\operatorname{lk}_\Lambda(w)\ast w$.
Being a cone vertex is intrinsic, so an isomorphism $\Delta\ast v\cong\Delta'\ast v'$ carries $v$ to a
cone vertex $w'$ of $\Delta'\ast v'$ and identifies $\Delta=\operatorname{lk}(v)$ with
$\operatorname{lk}(w')$. If $w'=v'$ this link is $\Delta'$; otherwise $w'\in c(\Delta')$ and it is
$\operatorname{lk}_{\Delta'}(w')\ast v'\cong\operatorname{lk}_{\Delta'}(w')\ast w'=\Delta'$.
\end{proof}

\subsection{Reconstruction}\label{subsec:reconstruction}

\begin{theorem}\label{thm:core}
Let $\Delta,\Delta'$ be fully exposed complexes on $[n]$ with
$\mathcal{R}(\Delta)=\mathcal{R}(\Delta')$. Then $\Delta=\Delta'$.
\end{theorem}

\begin{proof}
If $-\mathbf{e}_i$ lies in the common root set, then $i$ is a cone vertex of both and
\cref{lem:cone-roots}\ref{cr.2} reconstructs each; so assume neither has a cone vertex, whence
$n\geq2$ and all facets have at least two vertices by \cref{lem:private-vertex}. Put
$\mathcal{A}(\Delta''):=\{S\cup\{i\}: \emptyset\neq S\in\mathcal{A}_i(\Delta'')\}$, determined by
$\mathcal{R}(\Delta'')$ via \cref{lem:root-data}; we claim
$\mathcal{F}(\Delta)=\operatorname{Max}_\subseteq\mathcal{A}(\Delta)$, which suffices. A facet
$F$ has a private vertex $i$ (\cref{cor:fe-private}), so $F\in\mathcal{A}(\Delta)$; it is
maximal there, since $F\subsetneq S'\cup\{i'\}$ with $S'\in\mathcal{A}_{i'}(\Delta)$ would give a
facet $F''\supseteq S'\cup\{i'\}\supsetneq F$ by \cref{lem:closure-Ai}. Conversely, for
$T=S\cup\{i\}$ maximal in $\mathcal{A}(\Delta)$ and any facet $F''\supseteq S$, the same
inclusion gives $T\subseteq F''\in\mathcal{A}(\Delta)$, so $T=F''$.
\end{proof}

\begin{corollary}\label{cor:core-permuted}
Let $\Delta,\Delta'$ be fully exposed complexes on $[n]$ and suppose that
$\sigma\cdot\mathcal{R}(\Delta)=\mathcal{R}(\Delta')$ for some $\sigma\in S_n$. Then
$\sigma\cdot\Delta=\Delta'$; in particular $\Delta\cong\Delta'$ as abstract simplicial
complexes.
\end{corollary}

\begin{proof}
By \rchange{\cref{lem:R-equivariant}(ii)}, $\mathcal{A}_{\sigma(i)}(\sigma\cdot\Delta)=\sigma(\mathcal{A}_i(\Delta))$,
so $\sigma\cdot\Delta$ is again fully exposed, and
$\mathcal{R}(\sigma\cdot\Delta)=\sigma\cdot\mathcal{R}(\Delta)=\mathcal{R}(\Delta')$. Now apply
\cref{thm:core} to $\sigma\cdot\Delta$ and $\Delta'$.
\end{proof}

The fully exposed hypothesis cannot be dropped.

\begin{example}\label{ex:rigid-invisible}
On $[5]$ let $\Delta_1=P_5$ have facets $\{1,2\},\{2,3\},\{3,4\},\{4,5\}$ and let $\Delta_2$ be
$P_4$ on $\{1,2,4,5\}$ together with the isolated vertex $3$. Neither has a cone vertex, and
since all facets have at most two elements only singletons occur in the $\mathcal{A}_i$; for
both, the only inclusions $\mathcal{F}_j\subseteq\mathcal{F}_i$ with $j\neq i$ occur at
$(j,i)=(1,2)$ and $(5,4)$, so $\mathcal{A}_2=\{\{1\}\}$, $\mathcal{A}_4=\{\{5\}\}$ and the rest
are empty:
$$
\mathcal{R}(\Delta_1)=\mathcal{R}(\Delta_2)=\{k\mathbf{e}_1-\mathbf{e}_2,\ k\mathbf{e}_5-\mathbf{e}_4: k\geq1\}.
$$
\rchange{See \cref{fig:rigid-invisible}.} Both share the exposed facets $\{1,2\},\{4,5\}$,
but the hidden ones differ:
$\{2,3\},\{3,4\}$ versus $\{2,4\},\{3\}$. Hence $\Delta_1\neq\Delta_2$ with equal root sets.
In fact even $\Aut(X_{\Delta_1})\cong\Aut(X_{\Delta_2})$ as ind-groups
(\cref{prop:same-aut}).
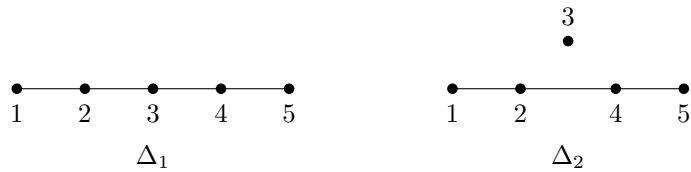
\begin{figure}[H]
\centering
\begin{tikzpicture}[scale=0.9, v/.style={circle,fill,inner sep=1.4pt}]
\begin{scope}
\draw (0,0)--(4,0);
\foreach \x/\l in {0/1,1/2,2/3,3/4,4/5} {\node[v,label=below:{\small $\l$}] at (\x,0) {};}
\node at (2,-1.0) {\small $\Delta_1$};
\end{scope}
\begin{scope}[xshift=6.4cm]
\draw (0,0)--(1,0)--(2.4,0)--(3.4,0);
\foreach \x/\l in {0/1,1/2,2.4/4,3.4/5} {\node[v,label=below:{\small $\l$}] at (\x,0) {};}
\node[v,label=above:{\small $3$}] at (1.7,0.7) {};
\node at (1.7,-1.0) {\small $\Delta_2$};
\end{scope}
\end{tikzpicture}
\caption{The complexes $\Delta_1$ and $\Delta_2$: equal root sets,
non-isomorphic complexes.}
\label{fig:rigid-invisible}
\end{figure}

\end{example}

\subsection{Examples and enumeration}\label{subsec:examples-combi}

\rchange{The reconstruction theorem (\cref{thm:core}) requires the fully exposed hypothesis, and
it is natural to ask how restrictive that condition is. We examine the standard families below
and show that, while fully exposed complexes are rare among graphs, the hypothesis loses no
homotopy type: every simplicial complex is homotopy equivalent to a fully exposed one
(\cref{prop:whisker-homotopy}).}

\subsubsection*{Graphs and skeleta}

The private-vertex criterion is decisive in the standard families. \rchange{Recall that a simple
graph on $[n]$ is a $1$-dimensional simplicial complex, its facets being the edges together
with the isolated vertices; a \emph{leaf} is a vertex of degree one; a \emph{star} is a graph
in which one vertex is adjacent to all the others and no other edges are present; and the
$k$-skeleton $\Delta^n_k$ of the simplex on $[n]$ is the complex of all faces of dimension at
most $k$.}

\begin{proposition}\label{prop:graph-exposed}
Let $\Delta$ be the $1$-dimensional complex of a simple graph $G$ on $[n]$, $n\geq2$. Then:
\begin{enumerate}[label=\textup{(\arabic*)}]
\item\label{gr.1} an edge $\{u,v\}$ is exposed \rchange{if and only if} $u$ or $v$ is a leaf;
\item\label{gr.2} $\Delta$ is fully exposed \rchange{if and only if} $G$ is a disjoint union of stars with no isolated vertex;
\item\label{gr.3} $\Delta$ is fully hidden \rchange{if and only if} $G$ has no vertex of degree exactly $1$;
\end{enumerate}
\end{proposition}

\rchange{\begin{proof}
\ref{gr.1} is \cref{lem:private-vertex}: $u$ is private for the edge $\{u,v\}$
precisely when $v$ lies in no other edge, i.e.\ when $v$ is a leaf. \ref{gr.2} follows from
\cref{cor:fe-private}, no two non-leaf vertices being adjacent exactly when $G$ is a disjoint
union of stars, together with the observation that an isolated vertex is a singleton facet,
hence hidden by \cref{lem:private-vertex}. For \ref{gr.3}, \cref{prop:rigid-cl} reduces the
question to whether some
$\mathcal{A}_i(\Delta)$ is nonempty. For a $1$-complex the members of the $\mathcal{A}_i$ are
$\emptyset$ and singletons, an edge being its own closure: one has $\{j\}\in\mathcal{A}_i$ if
and only if $j$ is a leaf with neighbour $i$, while an isolated vertex $j$ has
$\operatorname{cl}(\{j\})=\{j\}$ and therefore belongs to no $\mathcal{A}_i$; and
$\emptyset\in\mathcal{A}_i$ if and only if $i$ is a cone vertex, which for $n\geq2$ makes $G$ a
star and so again produces a leaf. Hence some $\mathcal{A}_i$ is nonempty if and only if $G$
has a vertex of degree exactly $1$.
\end{proof}}

Skeleta sit at the same extreme: for $0\leq k\leq n-2$ the $k$-skeleton $\Delta^n_k$ of the
simplex on $[n]$ is fully hidden. Indeed the link of a face $S$ is the $(k-|S|)$-skeleton of the
simplex on $[n]\setminus S$, whose facets are all the $(k+1-|S|)$-subsets of $[n]\setminus S$;
as $k+1-|S|<n-|S|$ this link has no cone vertex, so $\mathcal{A}_i(\Delta^n_k)=\emptyset$ for
every $i$ by \cref{rem:links} and \cref{prop:rigid-cl} applies.

Thus $\Delta_1,\Delta_2$ of \cref{ex:rigid-invisible} are $P_5$ and $P_4\sqcup\{3\}$, neither a
star forest.

\rchange{Fully exposed complexes are rare: in dimension one and without isolated vertices, they
are exactly the labelled star forests, a vanishing fraction of all graphs on $[n]$. The
hypothesis is therefore structural rather than generic. It is not a thin slice either: every
cone is fully exposed (\cref{lem:cone-roots}\ref{cr.3}), and no homotopy type is lost.}

\begin{proposition}\label{prop:whisker-homotopy}
Every simplicial complex is homotopy equivalent to a fully exposed one. Explicitly, let
$\mathrm{Wh}(\Delta)$ be obtained from $\Delta$ by adjoining a new vertex $v_F$ for each
$F\in\mathcal{F}(\Delta)$ and replacing $F$ by $F\cup\{v_F\}$. Then $\mathrm{Wh}(\Delta)$ is
fully exposed and $\mathrm{Wh}(\Delta)\simeq\Delta$.
\end{proposition}

\begin{proof}
The facets of $\mathrm{Wh}(\Delta)$ are the $F\cup\{v_F\}$ with $F\in\mathcal{F}(\Delta)$. As
$v_{F'}\notin F$ for $F'\neq F$, an inclusion $F\subseteq F'\cup\{v_{F'}\}$ would give
$F\subseteq F'$, contradicting maximality; hence $F\cup\{v_F\}$ is the unique facet containing
$F$, so $v_F$ is private and $\mathrm{Wh}(\Delta)$ is fully exposed by \cref{cor:fe-private}.
For the homotopy type, $\mathrm{Wh}(\Delta)=\Delta\cup\bigcup_{F}(v_F\ast F)$, where $v_F\ast F$
is the simplex on $F\cup\{v_F\}$. A simplex deformation retracts onto any of its faces, and the
retraction of $v_F\ast F$ onto $F$ is the identity on $F$. Since $v_F\notin\Delta$ and
$v_F\notin F'\cup\{v_{F'}\}$ for $F'\neq F$, the cone $v_F\ast F$ meets $\Delta$ and every other
cone exactly in a subcomplex of $F$, on which the retraction is the identity; the retractions
therefore glue to a deformation retraction of $\mathrm{Wh}(\Delta)$ onto $\Delta$.
\end{proof}

\noindent
Here $\mathrm{Wh}(\Delta)$ is a facetwise variant of the whiskering operation of Villarreal
\cite{Vil90}, generalised by Francisco--H\`a \cite{FrH08}, which attaches a pendant vertex to each
\emph{vertex} of the complex; here one new vertex is attached to each \emph{facet}\rchange{;
\cref{fig:whisker} shows the smallest interesting case}.\label{rem:fe-not-generic}

\begin{example}\label{ex:whisker-triangle}
Let $\Delta$ be the boundary of a triangle (\cref{fig:whisker}, left). Each facet of $\mathrm{Wh}(\Delta)$ acquires a new private vertex (drawn hollow), so the result is fully exposed and, by \cref{prop:whisker-homotopy}, still homotopy equivalent to a circle.
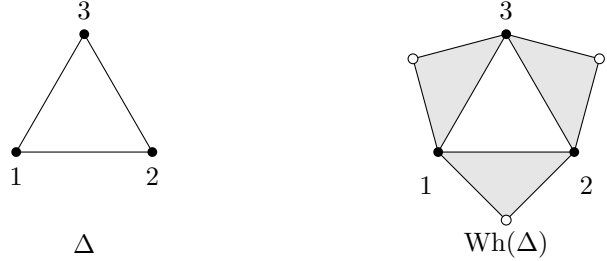
\begin{figure}[H]
\centering
\begin{tikzpicture}[scale=0.9, v/.style={circle,fill,inner sep=1.3pt},
                    w/.style={circle,draw,fill=white,inner sep=1.3pt}]
\begin{scope}
  \draw (0,0)--(2,0)--(1,1.732)--cycle;
  \node[v,label=below:{\small $1$}] at (0,0) {};
  \node[v,label=below:{\small $2$}] at (2,0) {};
  \node[v,label=above:{\small $3$}] at (1,1.732) {};
  \node at (1,-1.35) {\small $\Delta$};
\end{scope}
\begin{scope}[xshift=6.2cm]
  \fill[gray!20] (0,0)--(2,0)--(1,-1.0)--cycle;
  \fill[gray!20] (2,0)--(1,1.732)--(2.37,1.37)--cycle;
  \fill[gray!20] (1,1.732)--(0,0)--(-0.37,1.37)--cycle;
  \draw (0,0)--(1,-1.0)--(2,0);
  \draw (2,0)--(2.37,1.37)--(1,1.732);
  \draw (1,1.732)--(-0.37,1.37)--(0,0);
  \draw (0,0)--(2,0)--(1,1.732)--cycle;
  \node[v,label={[shift={(-0.16,-0.12)}]below:{\small $1$}}] at (0,0) {};
  \node[v,label={[shift={(0.16,-0.12)}]below:{\small $2$}}] at (2,0) {};
  \node[v,label=above:{\small $3$}] at (1,1.732) {};
  \node[w] at (1,-1.0) {}; \node[w] at (2.37,1.37) {}; \node[w] at (-0.37,1.37) {};
  \node at (1,-1.35) {\small $\mathrm{Wh}(\Delta)$};
\end{scope}
\end{tikzpicture}
\caption{Whiskering the boundary of a triangle.}
\label{fig:whisker}
\end{figure}
\end{example}

\section{Demazure roots and the automorphism group}\label{sec:aut}

\cref{thm:roots-agree} is the hinge of the paper: the degrees of the nonzero
$\ZZ^n$-homogeneous locally nilpotent derivations of $\KK[\Delta]$ are exactly the Demazure
roots of \cref{def:roots-comb}, so that every statement of \cref{sec:roots} becomes a statement
about $X_\Delta$. \cref{subsec:autgroup} then passes from individual $\GA$-actions to the
ind-group $\Aut(X_\Delta)$ as a whole.

\subsection{From derivations to roots}\label{subsec:der-roots}

Throughout this subsection $\mathcal{R}(\Delta)$ keeps the combinatorial meaning of
\cref{def:roots-comb}: until \cref{thm:roots-agree} we speak only of \emph{the degree of a
nonzero $\ZZ^n$-homogeneous locally nilpotent derivation} of $\KK[\Delta]$.

We write $\GA=(\KK,+)$ for the additive group of the field,
viewed as an algebraic group. A \emph{regular $\GA$-action} on
$X_\Delta$ is a morphism of algebraic varieties
$\GA\times X_\Delta\to X_\Delta$ satisfying the usual group
action axioms. Given a locally nilpotent derivation $\partial$
of $\KK[\Delta]$, the map $s\mapsto\exp(s\,\partial)$ defines
a regular $\GA$-action on $X_\Delta$, and every regular
$\GA$-action arises in this way
(cf.\ \cite[Section~1.2]{Lie10}, \cite{Fre17}, \cite[Remark~3.12]{DL24}). We say that
$X_\Delta$ is \emph{rigid} if it admits no nontrivial regular
$\GA$-action, and \emph{non-rigid} otherwise.\footnote{The word \emph{rigid} carries several incompatible
established meanings here. For Altmann--Christophersen \rchange{\cite{AC04,AC10}} a Stanley-Reisner scheme is
rigid when its first cotangent cohomology $T^{1}$ vanishes; for Freudenburg \cite{Fre17} and
Arzhantsev--Gaifullin \cite{AG17} an affine variety is rigid when it admits no nontrivial regular
$\GA$-action, which is the meaning fixed above and used throughout; and for simplicial complexes
there is the unrelated notion coming from bar-and-joint frameworks, central in face enumeration
since \cite{Kal87}. \emph{Flexible} is unavailable too, being in use for affine varieties whose
special automorphism group acts transitively on the smooth locus \cite{AFKKZ13}. This is why
\cref{sec:roots} speaks of \emph{hidden} and \emph{exposed} facets: ``rigid'' and ``non-rigid''
facets would clash with all three meanings and would suggest a property of $X_\Delta$ where one of
$\Delta$ is meant. We reserve \emph{rigid} for the variety.} By
\cref{lem:decomposition}, $X_\Delta$ is non-rigid if and only
if $\KK[\Delta]$ admits a nonzero $\ZZ^n$-homogeneous locally
nilpotent derivation.

We first classify the $\ZZ^n$-homogeneous derivations of $\KK[\Delta]$.\footnote{\cref{lem:homog-der} is not new: $\Der(\KK[\Delta])$ was computed in this generality by Brumatti--Simis \cite{BS95} (see also \cite{Tad09}), and the classification appears as \cite[Proposition~1]{Mul95}. The proof below only keeps \cref{subsec:der-roots} self-contained and free of hypotheses on $\dim\KK[\Delta]$; it is the statement that the rest of the paper uses. The complementary principle that every derivation of $\KK[\Delta]$ lifts to a derivation of $\KK[\mathbf{x}]$ preserving $I_\Delta$, on which \cref{lem:preserves-primes} and \cref{prop:roots-facets} are built, is \cite[Theorem~2.2]{DLMFR24}.} Write
$$
\Lambda\ :=\ \{a\in\ZZ^n_{\geq0}\ :\ \supp(a)\in\Delta\},
$$
so that $\{\overline{\mathbf{x}}^a\}_{a\in\Lambda}$ is the monomial basis of $\KK[\Delta]$
recalled in \cref{subsec:sr}; we use throughout the convention $\overline{\mathbf{x}}^b:=0$ for
$b\in\ZZ^n\setminus\Lambda$. Two features of $\Lambda$ drive everything: each graded piece
$\overline{B}_a$ is at most one-dimensional, and $\Lambda$ is a \emph{down-set}, i.e.\
$0\leq b\leq a$ \rchange{coordinatewise} and $a\in\Lambda$ imply $b\in\Lambda$ (because $\supp(b)\subseteq\supp(a)$ and
$\Delta$ is closed under subsets). We also write $N_\KK:=\operatorname{Hom}_\ZZ(\ZZ^n,\KK)$ with
basis $e_1^*,\dots,e_n^*$ dual to the standard basis of $\ZZ^n$, so that $e_i^*(a)=a_i$.

\begin{lemma}\label{lem:homog-der}
Let $\Delta$ be a simplicial complex on $[n]$ and let $\partial\neq0$ be a $\ZZ^n$-homogeneous
derivation of $\KK[\Delta]$ of degree $\alpha\in\ZZ^n$. Define $c_1,\dots,c_n\in\KK$ by
$\partial(\overline{x}_j)=c_j\,\overline{\mathbf{x}}^{\mathbf{e}_j+\alpha}$, with the
normalisation $c_j:=0$ when $\mathbf{e}_j+\alpha\notin\Lambda$, and put
$p:=\sum_{j}c_j\,e_j^*\in N_\KK$. Then:
\begin{enumerate}[label=\textup{(\arabic*)}]
\item\label{hd.1} $\alpha\in\ZZ^n_{\geq0}\cup\mathcal{N}$, where
$\mathcal{N}:=\{\beta-\mathbf{e}_i\ :\ i\in[n],\ \beta\in\ZZ^n_{\geq0},\ \beta_i=0\}$;
\item\label{hd.2} $\partial=\partial_{\alpha,p}$, that is,
$\partial(\overline{\mathbf{x}}^a)=p(a)\,\overline{\mathbf{x}}^{a+\alpha}$ for every $a\in\Lambda$;
\item\label{hd.3} if $\alpha=\beta-\mathbf{e}_i$ with $\beta\in\ZZ^n_{\geq0}$ and $\beta_i=0$,
then $p\in\KK^{*}\!\cdot e_i^*$, so that
$\partial=\lambda\,\overline{\mathbf{x}}^{\beta}\,\partial/\partial\overline{x}_i$ for some
$\lambda\in\KK^{*}$.
\end{enumerate}
\end{lemma}

\begin{proof}
Since $\partial$ is homogeneous of degree $\alpha$ we have
$\partial(\overline{x}_j)\in\overline{B}_{\mathbf{e}_j+\alpha}$, a space that is spanned by
$\overline{\mathbf{x}}^{\mathbf{e}_j+\alpha}$ if $\mathbf{e}_j+\alpha\in\Lambda$ and is zero
otherwise; so the scalars $c_j$ are well defined. A derivation is determined by its values on
algebra generators, so $\partial\neq0$ forces $c_i\neq0$ for some $i$, and then
$\mathbf{e}_i+\alpha\in\Lambda\subseteq\ZZ^n_{\geq0}$. Reading off coordinates,
$\alpha_l\geq0$ for every $l\neq i$ and $\alpha_i\geq-1$. This is \ref{hd.1}, and it yields the
sharper statement
\begin{equation}\label{eq:hd-star}
\text{if }\alpha\notin\ZZ^n_{\geq0}\text{, the index }i\text{ with }\alpha_i<0\text{ is unique,
}\alpha_i=-1\text{, and }c_l=0\text{ for all }l\neq i,
\end{equation}
the last assertion because $(\mathbf{e}_l+\alpha)_i=\alpha_i=-1<0$ for $l\neq i$, so that
$\mathbf{e}_l+\alpha\notin\Lambda$.

We prove \ref{hd.2}. Let $a\in\Lambda$. If $a+\alpha\notin\Lambda$ then
$\partial(\overline{\mathbf{x}}^a)\in\overline{B}_{a+\alpha}=0$ and
$\overline{\mathbf{x}}^{a+\alpha}=0$, so both sides vanish. Assume therefore
$a+\alpha\in\Lambda$ and induct on $|a|:=\sum_ja_j$. If $|a|=0$ then $a=0$,
$\partial(1)=0$ and $p(0)=0$. Let $|a|\geq1$, choose $j$ with $a_j\geq1$ and set
$a':=a-\mathbf{e}_j$, which lies in $\Lambda$ because $\Lambda$ is a down-set. The Leibniz rule
gives
$$
\partial(\overline{\mathbf{x}}^{a})\ =\ \partial(\overline{\mathbf{x}}^{a'})\,\overline{x}_j\ +\ \overline{\mathbf{x}}^{a'}\,\partial(\overline{x}_j).
$$
\emph{Case 1: $a'+\alpha\in\Lambda$.} By induction
$\partial(\overline{\mathbf{x}}^{a'})=p(a')\,\overline{\mathbf{x}}^{a'+\alpha}$, and
$\overline{\mathbf{x}}^{a'+\alpha}\overline{x}_j=\overline{\mathbf{x}}^{a+\alpha}\neq0$, so the
first summand is $p(a')\,\overline{\mathbf{x}}^{a+\alpha}$. For the second, if $c_j=0$ it is
$0$; and if $c_j\neq0$ then $\mathbf{e}_j+\alpha\in\Lambda$ and
$\overline{\mathbf{x}}^{a'}\,\overline{\mathbf{x}}^{\mathbf{e}_j+\alpha}=\overline{\mathbf{x}}^{a+\alpha}\neq0$,
so it is $c_j\,\overline{\mathbf{x}}^{a+\alpha}$. In both cases the second summand equals
$c_j\,\overline{\mathbf{x}}^{a+\alpha}$, and $p(a')+c_j=p(a)$ by linearity of $p$.

\emph{Case 2: $a'+\alpha\notin\Lambda$.} Here $a'+\alpha=(a+\alpha)-\mathbf{e}_j$ with
$a+\alpha\in\Lambda$; were $(a+\alpha)_j\geq1$ we would have
$0\leq a'+\alpha\leq a+\alpha$ and hence $a'+\alpha\in\Lambda$, since $\Lambda$ is a down-set.
Therefore $(a+\alpha)_j=0$, i.e.\ $\alpha_j=-a_j\leq-1$; by \eqref{eq:hd-star} this forces
$\alpha_j=-1$, hence $a_j=1$, and $c_l=0$ for every $l\neq j$, so that $p(a)=a_jc_j=c_j$. The
first summand vanishes, because $a'+\alpha\notin\Lambda$ gives
$\partial(\overline{\mathbf{x}}^{a'})\in\overline{B}_{a'+\alpha}=0$. For the second, note that
$\mathbf{e}_j+\alpha\leq a+\alpha$ (as $a\geq\mathbf{e}_j$) and $\mathbf{e}_j+\alpha\geq0$ (its
$j$-th coordinate is $1+\alpha_j=0$ and the others are $\alpha_l\geq0$), so
$\mathbf{e}_j+\alpha\in\Lambda$ and
$\overline{\mathbf{x}}^{a'}\,\overline{\mathbf{x}}^{\mathbf{e}_j+\alpha}=\overline{\mathbf{x}}^{a+\alpha}\neq0$.
Hence $\partial(\overline{\mathbf{x}}^{a})=c_j\,\overline{\mathbf{x}}^{a+\alpha}=p(a)\,\overline{\mathbf{x}}^{a+\alpha}$.

Finally \ref{hd.3}. If $\alpha=\beta-\mathbf{e}_i$ with $\beta\geq0$ and $\beta_i=0$ then
$\alpha_i=-1<0$, so \eqref{eq:hd-star} gives $c_l=0$ for $l\neq i$ and $p=c_i\,e_i^*$; and
$c_i\neq0$, since otherwise $p=0$ and $\partial=0$ by \ref{hd.2}. Setting $\lambda:=c_i$, part
\ref{hd.2} reads
$\partial(\overline{\mathbf{x}}^a)=\lambda\,a_i\,\overline{\mathbf{x}}^{a+\beta-\mathbf{e}_i}
=\lambda\,\overline{\mathbf{x}}^{\beta}\,\partial/\partial\overline{x}_i(\overline{\mathbf{x}}^a)$
for every $a\in\Lambda$.
\end{proof}

\begin{lemma}\label{lem:root-shape}
Let $\Delta$ be a simplicial complex on $[n]$. If $\alpha\in\ZZ^n$ is the degree of a nonzero $\ZZ^n$-homogeneous locally nilpotent derivation of $\KK[\Delta]$, then $\alpha\in\mathcal{N}$.
\end{lemma}

\begin{proof}
Let $\partial$ be a nonzero $\ZZ^n$-homogeneous locally nilpotent derivation of degree $\alpha$, and write $\partial=\partial_{\alpha,p}$ as in \cref{lem:homog-der}\ref{hd.2}. By \cref{lem:homog-der}\ref{hd.1}, $\alpha\in\ZZ^n_{\geq0}\cup\mathcal{N}$, so it suffices to show $\alpha\notin\ZZ^n_{\geq0}$. Assume for contradiction that $\alpha\in\ZZ^n_{\geq0}$, and set $S:=\supp(\alpha)$.

Since $\partial_{\alpha,p}\neq0$, there exists $a\in\ZZ^n_{\geq0}$ with $p(a)\neq0$ and $\supp(a)\cup S\in\Delta$. Using the linearity of $p$ and the formula $\partial_{\alpha,p}(\overline{\mathbf{x}}^a)=p(a)\,\overline{\mathbf{x}}^{a+\alpha}$, a direct induction gives
$$
\partial_{\alpha,p}^{\,m}(\overline{\mathbf{x}}^a)=\prod_{j=0}^{m-1}\bigl(p(a)+j\,p(\alpha)\bigr)\cdot\overline{\mathbf{x}}^{a+m\alpha}
$$
for all $m\geq1$. Since $\alpha\in\ZZ^n_{\geq0}$, we have $\supp(a+m\alpha)=\supp(a)\cup S$ for every $m\geq1$, which is a face of $\Delta$; hence $\overline{\mathbf{x}}^{a+m\alpha}\neq0$ for all $m\geq1$.

If $p(\alpha)\neq0$, take $a=\alpha$; this is legitimate, because $S=\supp(\alpha)$ satisfies $\supp(a_0)\cup S\in\Delta$ for the element $a_0$ chosen above, so in particular $S\in\Delta$ and $\supp(a)\cup S=S\in\Delta$; the product becomes $m!\,p(\alpha)^m$, which is nonzero in characteristic zero. If $p(\alpha)=0$, the product equals $p(a)^m\neq0$. In either case $\partial_{\alpha,p}^{\,m}(\overline{\mathbf{x}}^a)\neq0$ for all $m\geq1$, contradicting local nilpotency.
\end{proof}

\begin{corollary}\label{cor:root-form}
Every nonzero $\ZZ^n$-homogeneous locally nilpotent derivation of $\KK[\Delta]$ is of the form
$$
\partial_{\alpha,p}=\lambda\,\overline{\mathbf{x}}^{\beta}\,\frac{\partial}{\partial\overline{x}_i}
$$
for some $i\in[n]$, $\beta\in\ZZ^n_{\geq0}$ with $\beta_i=0$, and $\lambda\in\KK^*$.
\end{corollary}

\begin{proof}
By \cref{lem:root-shape}, $\alpha=\beta-\mathbf{e}_i\in\mathcal{N}$ for some $i\in[n]$ and $\beta\in\ZZ^n_{\geq0}$ with $\beta_i=0$; \cref{lem:homog-der}\ref{hd.3} then gives $p=\lambda\,e_i^*$ for some $\lambda\in\KK^*$ and $\partial_{\alpha,p}(\overline{\mathbf{x}}^a)=\lambda\,e_i^*(a)\,\overline{\mathbf{x}}^{a+\beta-\mathbf{e}_i}=\lambda\,a_i\,\overline{\mathbf{x}}^{a+\beta-\mathbf{e}_i}$, which coincides with $\lambda\,\overline{\mathbf{x}}^{\beta}\,\partial/\partial\overline{x}_i$ evaluated at $\overline{\mathbf{x}}^a$.
\end{proof}

\begin{lemma}\label{lem:preserves-primes}
Let $i\in[n]$, $\beta\in\ZZ^n_{\geq0}$ with $\beta_i=0$, and
set $S:=\supp(\beta)$. The derivation
$\partial_0:=\mathbf{x}^{\beta}\,\partial/\partial x_i$ of
$\KK[\mathbf{x}]$ satisfies
$\partial_0(\mathfrak{p}_F)\subseteq\mathfrak{p}_F$ for every
facet $F$ if and only if
$\mathcal{F}_S(\Delta)\subseteq\mathcal{F}_i(\Delta)$. In
that case, $\partial_0$ preserves $I_\Delta$ and descends to a
derivation of $\KK[\Delta]$.
\end{lemma}

\begin{proof}
Fix a facet $F$. The generators $x_l$, $l\notin F$, of
$\mathfrak{p}_F$ satisfy $\partial_0(x_l)=0$ unless $l=i$; if
$i\in F$ there is nothing to check, and if $i\notin F$ then
$\partial_0(x_i)=\mathbf{x}^\beta\in\mathfrak{p}_F$ if and
only if $S\not\subseteq F$. Hence
$\partial_0(\mathfrak{p}_F)\subseteq\mathfrak{p}_F$ for every
facet $F$ if and only if every facet containing $S$ also
contains $i$, i.e.,
$\mathcal{F}_S(\Delta)\subseteq\mathcal{F}_i(\Delta)$. Since
$I_\Delta=\bigcap_{F\in\mathcal{F}(\Delta)}\mathfrak{p}_F$,
the inclusion
$\partial_0(\mathfrak{p}_F)\subseteq\mathfrak{p}_F$ for every
$F$ gives $\partial_0(I_\Delta)\subseteq I_\Delta$, so
$\partial_0$ descends to a derivation of $\KK[\Delta]$.
\end{proof}

\begin{lemma}\label{lem:induced-derivation}
Let $i\in[n]$, $\beta\in\ZZ^n_{\geq0}$ with $\beta_i=0$, $S:=\supp(\beta)$, and suppose $\emptyset\neq\mathcal{F}_S(\Delta)\subseteq\mathcal{F}_i(\Delta)$. The induced derivation of $\partial:=\overline{\mathbf{x}}^{\beta}\,\partial/\partial\overline{x}_i$ on $\KK[F]$ is nonzero if and only if $F\in\mathcal{F}_S(\Delta)$.
\end{lemma}

\begin{proof}
First, $\partial$ is well defined on $\KK[F]$: by \cref{lem:preserves-primes} the lift $\mathbf{x}^\beta\,\partial/\partial x_i$ preserves $\mathfrak{p}_F$, so $\partial$ preserves $\mathfrak{q}_F$ and induces a derivation of $\KK[F]=\KK[\Delta]/\mathfrak{q}_F$. That derivation annihilates the class of $\overline{x}_j$ for $j\neq i$ and sends the class of $\overline{x}_i$ to that of $\overline{\mathbf{x}}^\beta$. The image of $\overline{\mathbf{x}}^\beta$ in $\KK[F]\cong\KK[x_j:j\in F]$ is nonzero if and only if $S\subseteq F$, i.e., $F\in\mathcal{F}_S(\Delta)$.
\end{proof}

\begin{proposition}\label{prop:roots-facets}
Let $i\in[n]$, let $\beta\in\ZZ^n_{\geq0}$ with $\beta_i=0$, and set $S:=\supp(\beta)$. The formula
$$
\partial:=\overline{\mathbf{x}}^{\beta}\,\frac{\partial}{\partial \overline{x}_i}
$$
defines a nonzero locally nilpotent derivation of $\KK[\Delta]$ if and only if
\begin{equation}\label{eq:root-condition}
\emptyset\neq\mathcal{F}_{S}(\Delta)\subseteq\mathcal{F}_i(\Delta).
\end{equation}
Condition \eqref{eq:root-condition} is the description of the negative roots in
\cite[Proposition~3]{Mul95}; see \cref{rem:muller-dictionary}.
\end{proposition}

\begin{proof}
Write $\partial_0:=\mathbf{x}^{\beta}\,\partial/\partial x_i$ on $\KK[\mathbf{x}]$. Since $\beta_i=0$ we have $\partial_0(x_j)\in\ker\partial_0$ for every $j$, whence $\partial_0$ is locally nilpotent on $\KK[\mathbf{x}]$ by an immediate induction on the $x_i$-degree.

Sufficiency: by \cref{lem:preserves-primes} the inclusion $\mathcal{F}_S(\Delta)\subseteq\mathcal{F}_i(\Delta)$ gives $\partial_0(\mathfrak{p}_F)\subseteq\mathfrak{p}_F$ for every facet $F$, hence $\partial_0(I_\Delta)\subseteq I_\Delta$, and $\partial_0$ descends to a locally nilpotent $\partial$ on $\KK[\Delta]$. As $\mathcal{F}_S(\Delta)\neq\emptyset$ we have $S\in\Delta$, so $\overline{\mathbf{x}}^{\beta}$ is a basis element and $\partial(\overline{x}_i)=\overline{\mathbf{x}}^{\beta}\neq0$.

Conversely, we prove the contrapositive: if \eqref{eq:root-condition} fails, then the formula does not define a nonzero locally nilpotent derivation of $\KK[\Delta]$. The condition can fail in two ways: either $\mathcal{F}_S(\Delta)\not\subseteq\mathcal{F}_i(\Delta)$, or $\mathcal{F}_S(\Delta)=\emptyset$. Suppose first that $\mathcal{F}_S(\Delta)\not\subseteq\mathcal{F}_i(\Delta)$. Pick a facet $F$ with $S\subseteq F$ and $i\notin F$. Since $F$ is a facet, $F\cup\{i\}\notin\Delta$, so there is a minimal non-face $\sigma\subseteq F\cup\{i\}$, and necessarily $i\in\sigma$, for otherwise $\sigma\subseteq F$ would be a face. Then $\mathbf{x}^\sigma\in\operatorname{G}(I_\Delta)\subseteq I_\Delta$, while $\partial_0(\mathbf{x}^\sigma)=\mathbf{x}^{\beta+\mathbf{e}_{\sigma\setminus\{i\}}}$ has support $S\cup(\sigma\setminus\{i\})\subseteq F$, a face; hence $\partial_0(\mathbf{x}^\sigma)\notin I_\Delta$. Now suppose $\mathcal{F}_S(\Delta)=\emptyset$, i.e.\ $S\notin\Delta$. Then $\partial_0$ does preserve $I_\Delta$ (each nonzero $\partial_0(\mathbf{x}^\sigma)$ is divisible by $\mathbf{x}^{\mathbf{e}_S}\in I_\Delta$), but the induced derivation of $\KK[\Delta]$ is zero: it annihilates $\overline{x}_j$ for $j\neq i$ and sends $\overline{x}_i$ to $\overline{\mathbf{x}}^\beta=0$.
\end{proof}

We can now identify the combinatorial root set of \cref{def:roots-comb} with the set of degrees
of homogeneous locally nilpotent derivations, justifying the terminology of \cref{sec:roots}.

\begin{theorem}\label{thm:roots-agree}\label{cor:roots-formula}
Let $\Delta$ be a simplicial complex on $[n]$. A vector $\alpha\in\ZZ^n$ is the degree of a
nonzero $\ZZ^n$-homogeneous locally nilpotent derivation of $\KK[\Delta]$ if and only if
$\alpha$ is a Demazure root of $\Delta$ in the sense of \cref{def:roots-comb}. Equivalently,
$$
\mathcal{R}(\Delta)=\bigl\{\beta-\mathbf{e}_i \;:\; i\in[n],\ \beta\in\ZZ^n_{\geq0},\ \beta_i=0,\ \emptyset\neq\mathcal{F}_{\supp(\beta)}(\Delta)\subseteq\mathcal{F}_i(\Delta)\bigr\}.
$$
From here on the two readings of $\mathcal{R}(\Delta)$ may be used interchangeably.
\end{theorem}

\begin{proof}
By \cref{cor:root-form}, every nonzero $\ZZ^n$-homogeneous locally nilpotent derivation of
$\KK[\Delta]$ has the form $\lambda\,\overline{\mathbf{x}}^{\beta}\,\partial/\partial\overline{x}_i$
with $i\in[n]$, $\beta\in\ZZ^n_{\geq0}$, $\beta_i=0$ and $\lambda\in\KK^*$, and its degree is
$\beta-\mathbf{e}_i$. By \cref{prop:roots-facets}, such a formula defines a nonzero locally
nilpotent derivation of $\KK[\Delta]$ if and only if
$\emptyset\neq\mathcal{F}_{\supp(\beta)}(\Delta)\subseteq\mathcal{F}_i(\Delta)$, which proves
the displayed description. By \cref{lem:closure-Ai}, that condition is equivalent
to $\supp(\beta)\in\mathcal{A}_i(\Delta)$, which is the defining condition of
\cref{def:roots-comb}.
\end{proof}

\rchange{\begin{corollary}\label{cor:hidden-rigid}
Let $\Delta$ be a simplicial complex on $[n]$. The following are equivalent:
\begin{enumerate}[label=\textup{(\alph*)}]
\item\label{hr.a} $X_\Delta$ is rigid;
\item\label{hr.b} $\Delta$ is fully hidden;
\item\label{hr.c} every face of $\Delta$ is an intersection of facets.
\end{enumerate}
\end{corollary}}%
\rcomment{sugiero reescribir el corolario.}

\rchange{\begin{proof}
By \cref{lem:decomposition}, $X_\Delta$ is rigid if and only if $\KK[\Delta]$ admits no nonzero
$\ZZ^n$-homogeneous locally nilpotent derivation. By \cref{thm:roots-agree}, this holds if and
only if $\mathcal{R}(\Delta)=\emptyset$. \cref{prop:rigid-cl} gives that
$\mathcal{R}(\Delta)=\emptyset$ if and only if $\Delta$ is fully hidden, so \ref{hr.a} and
\ref{hr.b} are equivalent; the same proposition shows that $\Delta$ is fully hidden if and only
if every face of $\Delta$ is an intersection of facets, so \ref{hr.b} and \ref{hr.c} are
equivalent.
\end{proof}}%
\rcomment{La prueba tambi\'en como que se pone a difariar. Sugiero la prueba siguiente.}

\begin{remark}\label{rem:hidden-invisible}
Concretely, an exposed facet $F$, say with $S\in\mathcal{A}_i(\Delta)$ and
$S\subseteq F$, makes $\overline{\mathbf{x}}^{\mathbf{e}_S}\,\partial/\partial\overline{x}_i$ a
nonzero locally nilpotent derivation (\cref{prop:roots-facets}); conversely, by
\cref{lem:induced-derivation} every such derivation induces the zero derivation on $\KK[F]$ for
each hidden facet $F$, so when all facets are hidden it vanishes, $\bigcap_{F}\mathfrak{q}_F$
being $0$. This is the source of the phenomenon in \cref{ex:rigid-invisible}: hidden facets are
invisible to the $\GA$-actions on $X_\Delta$.
\end{remark}

\subsection{The automorphism ind-group}\label{subsec:autgroup}

Recall that $\Aut(X_\Delta)$ carries a natural ind-group structure (\cref{subsec:ind}). We write $T_0:=(\KK^*)^n$ for the standard torus, embedded in $\Aut(X_\Delta)$ via $t\cdot x_i=t_ix_i$, and identify $X(T_0)$ with $\ZZ^n$. Two facts proved here make $\Aut(X_\Delta)$ usable as an invariant: $T_0$ is a maximal torus and all maximal tori are conjugate (\cref{thm:tori}), and the weights of the root subgroups relative to $T_0$ are exactly the Demazure roots (\cref{prop:root-subgroups}); between them, an abstract isomorphism of ind-groups can be read off on the character lattice, which is what \cref{sec:main} does. We also show that the fully exposed hypothesis is sharp (\cref{prop:same-aut}) and settle the opposite extreme (\cref{thm:rigid-aut}).
\begin{lemma}\label{lem:origin}
Every $\varphi\in\Aut(X_\Delta)$ permutes the irreducible components
$\{V_F: F\in\mathcal{F}(\Delta)\}$, with $|F|=|\varphi(F)|$ for the induced
permutation of facets. Consequently $\varphi$ preserves the cone locus
$$
Y_\Delta:=\bigcap_{F\in\mathcal{F}(\Delta)}V_F
=\{x\in X_\Delta: x_l=0\text{ for every }l\notin c(\Delta)\}
\cong\AF^{|c(\Delta)|},
$$
where $c(\Delta)\subseteq[n]$ denotes the set of cone vertices of $\Delta$.
In particular, if $\Delta$ has no cone vertex, then every automorphism fixes
the origin $0\in X_\Delta$, and hence preserves the maximal ideal
$\mathfrak{m}_0:=(\overline{x}_1,\dots,\overline{x}_n)\subseteq\KK[\Delta]$
and all of its powers.
\end{lemma}

\begin{proof}
An automorphism of a variety permutes its irreducible components and preserves their
dimensions, which gives the first assertion; consequently $\varphi$ preserves
$\bigcap_{F}V_F=V(\sum_{F}\mathfrak{q}_F)$, which by \eqref{eq:intersections} with
$H=c(\Delta)$ is the coordinate subspace $\{x_l=0: l\notin c(\Delta)\}\cong\AF^{|c(\Delta)|}$,
reduced because the quotient in \eqref{eq:intersections} is a polynomial ring.
If $\Delta$ has no cone vertex, then $Y_\Delta=\{0\}$, so $\varphi(0)=0$;
the comorphism $\varphi^*$ then preserves
$\mathfrak{m}_0=\{f\in\KK[\Delta]: f(0)=0\}$ and hence $\mathfrak{m}_0^k$
for every $k\geq 1$.
\end{proof}

\begin{lemma}\label{lem:components}
Let $G\subseteq\Aut(X_\Delta)$ be a connected algebraic subgroup. Then 
$G$ maps $V_F$ onto itself for every $F\in\mathcal{F}(\Delta)$; in 
particular, $G$ acts algebraically on the cone locus $Y_\Delta$.
\end{lemma}

\begin{proof}
Fix $F$ and write $\varphi^*$ for the comorphism of $\varphi\in\Aut(X_\Delta)$. For a facet
$G'$ one has $\varphi(V_F)=V_{G'}$ if and only if
$\varphi^*(\mathfrak{q}_{G'})\subseteq\mathfrak{q}_F$: the inclusion says that every function
vanishing on $V_{G'}$ pulls back to one vanishing on $V_F$, i.e.\ $\varphi(V_F)\subseteq
V_{G'}$, and since $\varphi(V_F)$ is itself a component (\cref{lem:origin}) while distinct
components are incomparable, this forces equality. Since $\KK[\Delta]$ is a rational $G$-module, the $G$-orbit of each
generator $\overline{x}_l$ spans a finite-dimensional subspace
$W_l\subseteq\KK[\Delta]$, and the map $g\mapsto g^*(\overline{x}_l)\in W_l$
is a morphism of varieties. Since $\mathfrak{q}_F\cap W_l$ is a linear
subspace of $W_l$, the set $\{g\in G: g^*(\overline{x}_l)\in\mathfrak{q}_F\}$
is closed in $G$ for each $l$; hence
$$
Z_{G'}:=\bigcap_{l\notin G'}\{g\in G: g^*(\overline{x}_l)\in\mathfrak{q}_F\}
$$
is closed in $G$ as a finite intersection of closed sets. The sets $Z_{G'}$,
$G'\in\mathcal{F}(\Delta)$, are pairwise disjoint and cover $G$; as $G$ is
connected and $e\in Z_F$, we conclude $G=Z_F$. Finally, $Y_\Delta=
\bigcap_F V_F$ is $G$-stable.
\end{proof}
\begin{lemma}\label{lem:dim}
Let $y_0\in Y_\Delta$ and let $\mathfrak{m}_{y_0}\subseteq\KK[\Delta]$ be 
the corresponding maximal ideal. Then 
$\dim_\KK\mathfrak{m}_{y_0}/\mathfrak{m}_{y_0}^2=n$ and 
$\bigcap_{k\geq1}\mathfrak{m}_{y_0}^k=0$.
\end{lemma}

\begin{proof}
Write $c:=c(\Delta)$ and let $a\in\KK^n$ be the coordinate vector of $y_0$,
so that $a_l=0$ for $l\notin c$. For the first assertion, let 
$M:=(x_1-a_1,\dots,x_n-a_n)\subseteq\KK[\mathbf{x}]$, so that 
$\mathfrak{m}_{y_0}=M/I_\Delta$ and 
$\mathfrak{m}_{y_0}/\mathfrak{m}_{y_0}^2=M/(M^2+I_\Delta)$. We claim 
$I_\Delta\subseteq M^2$. For every $\sigma\notin\Delta$, one has 
$|\sigma\setminus c|\geq2$: if $\sigma\setminus c=\emptyset$ then 
$\sigma\subseteq c$ is in every facet, and if $\sigma\setminus c=\{l\}$ 
then $\sigma\subseteq c\cup\{l\}\subseteq F$ for any facet $F$ containing 
$l$; in both cases $\sigma\in\Delta$, a contradiction. Hence $\mathbf{x}^\sigma$ 
is divisible by $x_lx_{l'}$ for two distinct $l,l'\notin c$, and since 
$x_l-a_l\in M$ and likewise for $x_{l'}$, we get $\mathbf{x}^\sigma\in M^2$. 
Therefore $\mathfrak{m}_{y_0}/\mathfrak{m}_{y_0}^2=M/M^2$ has dimension $n$.

For the second assertion, let $f\in\bigcap_k\mathfrak{m}_{y_0}^k$. By the 
Krull intersection theorem \cite[Corollary~5.4]{Eis95}, there is 
$t\in\mathfrak{m}_{y_0}$ with $(1-t)f=0$. Every irreducible component 
$V_F$ contains $Y_\Delta$, hence $y_0$; since $(1-t)(y_0)=1\neq0$, the 
component $V_F$ is not contained in the zero set of $1-t$, and from 
$V_F\subseteq\{f=0\}\cup\{1-t=0\}$ and the irreducibility of $V_F$ we 
conclude $V_F\subseteq\{f=0\}$. As this holds for every component and 
$\KK[\Delta]$ is reduced, $f=0$.
\end{proof}

The following general lemma on ind-groups will be used repeatedly, here and in \cref{sec:main}.

\begin{lemma}\label{lem:image-algebraic}
Let $\theta\colon G\to G'$ be an injective homomorphism of ind-groups, and let $A\subseteq G$ be a connected algebraic subgroup. Then $\theta(A)$ is a connected algebraic subgroup of $G'$, commutative if $A$ is, and $\theta|_A\colon A\to\theta(A)$ is an isomorphism of algebraic groups.
\end{lemma}

\begin{proof}
Since $A$ is an algebraic subgroup of $G$, it is contained in some filtration
piece $G_k$. Since $\theta$ is a morphism of ind-groups,
$\theta(A)\subseteq G'_\ell$ for some $\ell$, so $\theta(A)$ is an abstract
subgroup of $G'$ that is a constructible subset of the algebraic variety
$G'_\ell$ (by Chevalley's theorem on images of morphisms). Its closure
$Z:=\overline{\theta(A)}$ in $G'_\ell$ is a closed subgroup: since $G'$ is an
ind-group, translations and inversion are isomorphisms of ind-varieties, hence
homeomorphisms, so
$hZ=h\overline{\theta(A)}=\overline{h\theta(A)}=\overline{\theta(A)}=Z$ for
every $h\in\theta(A)$, and $Z^{-1}=\overline{\theta(A)^{-1}}=Z$. Hence
$\theta(A)Z=Z$, and for every $z\in Z$ right translation by $z$ gives
$Zz=\overline{\theta(A)}\,z=\overline{\theta(A)z}\subseteq\overline{Z}=Z$, so
$ZZ\subseteq Z$ and $Z$ is a closed subgroup of $G'_\ell$. Since $\theta(A)$ is
constructible and dense in the noetherian space $Z$, it contains a nonempty open
subset of $Z$ \cite[Proposition~AG.1.3]{Bor91}; then $\theta(A)$, as a union of
$\theta(A)$-translates of that open subset, is open in $Z$. An open subgroup of
a topological group is closed, so $\theta(A)=Z$ is a closed algebraic subgroup
of $G'_\ell$, hence an algebraic subgroup of $G'$.

It is connected and commutative if $A$ is, being the image of $A$ under a group
homomorphism. Since $\theta$ is injective, $\theta|_A\colon A\to\theta(A)$ is a
bijective homomorphism of algebraic groups, hence an isomorphism in
characteristic zero \cite[Corollary~5.3.3]{Spr98}.
\end{proof}
The following result is standard for irreducible affine varieties with a torus action
(\cite[Section~1.2 and Theorem~2.7]{Lie10}, \cite[Proposition~4.3]{LRU23}, and
\cite[Lemma~4.9]{DLMFR24} for finite-dimensional automorphism groups of monomial algebras).
Since those references assume either irreducibility or finite-dimensionality of the
automorphism group, neither of which holds here, we include the proof.
\begin{proposition}\label{prop:root-subgroups}
The weights of the root subgroups of $\Aut(X_\Delta)$ with respect to $T_0$ are exactly the Demazure roots $\alpha\in\mathcal{R}(\Delta)$.
\end{proposition}

\begin{proof}
Let $\alpha\in\mathcal{R}(\Delta)$.
Since $\mathcal{R}(\Delta)$ is defined combinatorially
(\cref{def:roots-comb}), we must first convert $\alpha$ into an actual derivation: by
\cref{thm:roots-agree} we may write $\alpha=\beta-\mathbf{e}_i$ with $\beta\in\ZZ^n_{\geq0}$,
$\beta_i=0$, in such a way that
$\partial:=\lambda\,\overline{\mathbf{x}}^\beta\,\partial/\partial\overline{x}_i$ is a nonzero
$\ZZ^n$-homogeneous locally nilpotent derivation of $\KK[\Delta]$ of degree $\alpha$, for every
$\lambda\in\KK^*$.
Since $\beta_i=0$, one has $\partial^2(\overline{x}_j)=0$ for every
generator $\overline{x}_j$, so
$u_s:=\exp(s\partial)$ satisfies
$u_s(\overline{x}_j)=\overline{x}_j+s\lambda\delta_{ij}\,
\overline{\mathbf{x}}^\beta$
and defines a $\GA$-action.
Set $d:=\max(1,|\beta|)$ and let $A_{\leq d}$ denote the subspace
of $A=\KK[\Delta]$ spanned by monomials of degree at most $d$.
The map $\iota\colon s\mapsto(u_s,u_{-s})$ is an injective
affine-linear map $\GA\to\bigl((A_{\leq d})^n\bigr)^2$
(only the coefficient of $\overline{\mathbf{x}}^{\beta}$ in the
$i$-th component varies), hence a closed immersion.
Since $u_s^{-1}=u_{-s}$, the image lies in the filtration piece
$\Aut(X_\Delta)\cap\bigl((A_{\leq d})^n\bigr)^2$ and is closed
there; thus $U:=\{u_s : s\in\KK\}$ is a closed algebraic subgroup
of $\Aut(X_\Delta)$ isomorphic to $\GA$.
For $t\in T_0$ and a monomial $\overline{\mathbf{x}}^a$,
$$
(t\circ\partial\circ t^{-1})(\overline{\mathbf{x}}^a)
=t^{-a}\,t\cdot\partial(\overline{\mathbf{x}}^a)
=t^{\alpha}\,\partial(\overline{\mathbf{x}}^a),
$$
so $t\partial t^{-1}=t^\alpha\partial$ and hence
$t\,u_s\,t^{-1}=\exp(s\,t^\alpha\partial)=u_{t^\alpha s}$;
as $\alpha\neq0$, $U$ is a root subgroup of weight $\alpha$.

Conversely, let $(U,\varepsilon)$ be a root subgroup of weight
$\alpha$. Since $U$ is an algebraic subgroup of $\Aut(X_\Delta)$,
it acts algebraically on $X_\Delta$ (see \cref{subsec:ind}),
so $\varepsilon$ defines a $\GA$-action, which by the correspondence
between $\GA$-actions and locally nilpotent derivations
(cf.\ \cite[Remark~3.12]{DL24}) equals
$\varepsilon(s)=\exp(s\partial)$ for a nonzero locally nilpotent
derivation $\partial$.
For $t\in T_0$, the root subgroup condition gives
$t\varepsilon(s)t^{-1}=\varepsilon(t^\alpha s)$ for all $s$,
i.e.,
$\exp(s\,t\partial t^{-1})=\exp(s\,t^\alpha\partial)$;
comparing the coefficients of $s$ in
$\exp(s\,\cdot\,)(f)=f+s(\cdot)(f)+\cdots$
yields
$t\partial t^{-1}=t^\alpha\partial$ for all $t\in T_0$.
Decomposing $\partial=\sum_\gamma\partial_\gamma$ into
$\ZZ^n$-homogeneous components
(see \cref{subsec:der}) and using
$t\partial_\gamma t^{-1}=t^\gamma\partial_\gamma$, we obtain
$\sum_\gamma t^\gamma\partial_\gamma
=\sum_\gamma t^\alpha\partial_\gamma$ for all $t$;
since distinct characters of $T_0$ are linearly independent,
$\partial_\gamma=0$ for every $\gamma\neq\alpha$,
so $\partial=\partial_\alpha$ is homogeneous of degree $\alpha$.
By \cref{cor:root-form}, $\partial$ has the
form $\lambda'\overline{\mathbf{x}}^{\beta'}\,
\partial/\partial\overline{x}_{i'}$, and $\alpha\in\mathcal{R}(\Delta)$
by \cref{thm:roots-agree}.
\end{proof}

The following elementary combinatorial lemma will be used to compare two Stanley-Reisner presentations of the same algebra.

\begin{lemma}\label{lem:matching}
Let $(A_1,\dots,A_r)$ and $(B_1,\dots,B_r)$ be lists of subsets of $[n]$ with $\bigcup_t A_t=\bigcup_t B_t=[n]$ and
$$
\Bigl|\bigcap_{t\in I}A_t\Bigr|=\Bigl|\bigcap_{t\in I}B_t\Bigr|
\qquad\text{for every nonempty } I\subseteq[r].
$$
Then there is a permutation $\sigma\in S_n$ with $\sigma(A_t)=B_t$ for every $t$.
\end{lemma}

\begin{proof}
For $v\in[n]$ let $s_A(v):=\{t : v\in A_t\}$, which is nonempty by the covering assumption, and define $s_B$ likewise. For every nonempty $J\subseteq[r]$, inclusion--exclusion gives
$$
\#\{v : s_A(v)=J\}=\sum_{K\supseteq J}(-1)^{|K\setminus J|}\Bigl|\bigcap_{t\in K}A_t\Bigr|,
$$
and the same formula for $B$; by hypothesis the two counts agree for every $J$. Choosing a bijection $\sigma$ of $[n]$ that matches, for each $J$, the set $\{v: s_A(v)=J\}$ with $\{v: s_B(v)=J\}$, we get $s_B(\sigma(v))=s_A(v)$ for all $v$, that is, $v\in A_t$ if and only if $\sigma(v)\in B_t$, i.e., $\sigma(A_t)=B_t$.
\end{proof}

\begin{theorem}\label{thm:tori}
Let $\Delta$ be a simplicial complex on $[n]$.
\begin{enumerate}[label=\textup{(\alph*)}]
\item\label{tori.a} The standard torus $T_0=(\KK^*)^n$ embeds into $\Aut(X_\Delta)$ as a closed algebraic subgroup.
\item\label{tori.b} Every algebraic torus $D\subseteq\Aut(X_\Delta)$ satisfies $\dim D\leq n$.
\item\label{tori.c} Every algebraic torus of dimension $n$ in $\Aut(X_\Delta)$ is conjugate to $T_0$.
\end{enumerate}
\end{theorem}
\begin{proof}
We first prove~\ref{tori.a}. The torus $T_0=(\KK^*)^n$ acts on $\KK[\mathbf{x}]$ by $t\cdot x_l=t_lx_l$; this action preserves every monomial ideal, and since $I_\Delta$ has no linear generators the induced action on $\KK[\Delta]$ is faithful, yielding an embedding $T_0\hookrightarrow\Aut(X_\Delta)$. In the coordinates of \cref{subsec:ind}, the image of $t=(t_1,\ldots,t_n)\in T_0$ is the pair $(\varphi_t,\varphi_t^{-1})$ where $\varphi_t(\overline{x}_l)=t_l\overline{x}_l$ and $\varphi_t^{-1}(\overline{x}_l)=t_l^{-1}\overline{x}_l$; this is cut out by polynomial equations inside each filtration piece, hence closed.

We now show~\ref{tori.b}. Let $D\subseteq\Aut(X_\Delta)$ be an algebraic torus. Being connected, $D$ preserves every irreducible component of $X_\Delta$ (\cref{lem:components}) and therefore acts algebraically on the cone locus $Y_\Delta\cong\AF^{|c(\Delta)|}$ of \cref{lem:origin}. By \cite[Theorem~1]{BB66}, this action has a fixed point $y_0\in Y_\Delta\subseteq X_\Delta$; when $\Delta$ has no cone vertex, $Y_\Delta=\{0\}$ and necessarily $y_0=0$. Write $\mathfrak{m}:=\mathfrak{m}_{y_0}$ for the maximal ideal of $y_0$; by \cref{lem:dim}, $\dim_\KK\mathfrak{m}/\mathfrak{m}^2=n$ and $\bigcap_k\mathfrak{m}^k=0$. Since $y_0$ is fixed by $D$, the ideals $\mathfrak{m}$ and $\mathfrak{m}^2$ are $D$-stable, so $D$ acts on $V:=\mathfrak{m}/\mathfrak{m}^2$, and we obtain a homomorphism $\rho\colon D\to\GL(V)$.

We claim that $\rho$ is injective. Set $z_l:=\overline{x}_l - \overline{x}_l(y_0)\in\mathfrak{m}$, so that $z_1,\ldots,z_n$ generate $\mathfrak{m}$ and span $V$ modulo $\mathfrak{m}^2$. Let $\sigma\in D$ with $\rho(\sigma)=\operatorname{id}$, i.e., $\sigma(z_l)\equiv z_l\pmod{\mathfrak{m}^2}$. Writing $\sigma(z_l)=z_l+g_l$ with $g_l\in\mathfrak{m}^2$ and expanding products, $\sigma$ acts trivially on each quotient $\mathfrak{m}^k/\mathfrak{m}^{k+1}$; hence $\sigma$ is unipotent on the finite-dimensional $D$-invariant quotient $\KK[\Delta]/\mathfrak{m}^k$ for every $k$. Since $D$ is a torus acting rationally, $\sigma$ is also semisimple there, so $\sigma=\operatorname{id}$ on $\KK[\Delta]/\mathfrak{m}^k$ for all $k$. As $\bigcap_k\mathfrak{m}^k=0$ (\cref{lem:dim}), we conclude $\sigma=\operatorname{id}$. Therefore $\rho(D)$ is a torus in $\GL_n(\KK)$ and $\dim D\leq n$.

We finally prove~\ref{tori.c}, following \cite[Theorem~2]{BB66} (see also \cite[Proposition~10.5.4]{FK18} for the normal case); since $X_\Delta$ is reducible, we include the details. Let $D$ be a torus of dimension $n$; by~\ref{tori.b} it is maximal. Since $D$ is an algebraic subgroup, $\KK[\Delta]$ is a rational $D$-module, that is, every element of $\KK[\Delta]$ belongs to a finite-dimensional $D$-stable subspace (see \cite[Proposition~2.3.6]{Spr98}); as $D$ is a torus, $\KK[\Delta]$ decomposes into weight spaces $\KK[\Delta]=\bigoplus_{\lambda\in X(D)}\KK[\Delta]_\lambda$, compatibly with the algebra structure (\cite[Theorem~3.2.3]{Spr98}). The fixed point $y_0\in Y_\Delta$ and the maximal ideal $\mathfrak{m}:=\mathfrak{m}_{y_0}$ are as in the proof of~\ref{tori.b}. Diagonalizing $\rho(D)$ in $\GL(V)$, we obtain a basis $\overline{y}_1,\ldots,\overline{y}_n$ of $V=\mathfrak{m}/\mathfrak{m}^2$ consisting of $D$-eigenvectors with weights $\chi_1,\ldots,\chi_n\in X(D)$; since $\rho$ is an isomorphism onto a maximal torus of $\GL_n(\KK)$, the characters $\chi_1,\ldots,\chi_n$ form a $\ZZ$-basis of $X(D)$. Since every rational representation of a diagonalizable group is a direct sum of one-dimensional representations (\cite[Theorem~3.2.3(c)]{Spr98}), the exact sequence of rational $D$-modules $0\to\mathfrak{m}^2\to\mathfrak{m}\to V\to0$ splits, and the $\overline{y}_i$ lift to $D$-eigenvectors $y_i\in\mathfrak{m}$ of weight $\chi_i$.

From here the identity $\KK[\Delta]=\KK[y_1,\dots,y_n]$ is the standard argument of
Bia\l ynicki-Birula \cite[Theorem~2]{BB66} (see \cite[Proposition~10.5.4]{FK18}): the weights of
$\KK[\Delta]$ lie in the pointed cone $C:=\ZZ_{\geq0}\chi_1+\dots+\ZZ_{\geq0}\chi_n$ and
$\KK[\Delta]_0=\KK$, and induction on $|\lambda|:=\sum_ia_i$ for $\lambda=\sum_ia_i\chi_i\in C$
gives $\mathfrak{m}_\lambda\subseteq\KK[y_1,\dots,y_n]$, whence
$\KK[\Delta]=\KK\oplus\mathfrak{m}=\KK[y_1,\dots,y_n]$. It uses only
$\dim_\KK\mathfrak{m}/\mathfrak{m}^2=n$, $\bigcap_k\mathfrak{m}^k=0$ (\cref{lem:dim}) and the
rationality of the $D$-module $\KK[\Delta]$, none of which involves irreducibility. What is
specific to our setting begins now.

The presentation by the $y_i$ is again of Stanley-Reisner type. Let $\widetilde{\Phi}\colon\KK[\mathbf{x}]\twoheadrightarrow\KK[\Delta]$, $x_i\mapsto y_i$, surjective by the previous paragraph. Grade $\KK[\mathbf{x}]$ by $X(D)$ with $\deg x_i:=\chi_i$; since the $\chi_i$ form a $\ZZ$-basis, every graded piece of $\KK[\mathbf{x}]$ is either zero or the line spanned by a single monomial. The map $\widetilde{\Phi}$ is graded, so $\ker\widetilde{\Phi}$ is a graded, hence monomial, ideal. Moreover $\KK[\mathbf{x}]/\ker\widetilde{\Phi}\cong\KK[\Delta]$ is reduced, so $\ker\widetilde{\Phi}$ is radical: if $\mathbf{x}^a\in\ker\widetilde{\Phi}$ and $S=\supp(a)$, then $\mathbf{x}^{m\mathbf{e}_S}=\mathbf{x}^{m\mathbf{e}_S-a}\mathbf{x}^a\in\ker\widetilde{\Phi}$ for $m:=\max_ia_i$, hence $\mathbf{x}^{\mathbf{e}_S}\in\ker\widetilde{\Phi}$. Setting $\Gamma:=\{A\subseteq[n] : \mathbf{y}^{\mathbf{e}_A}:=\prod_{l\in A}y_l\neq0\}$, we verify that $\Gamma$ is a simplicial complex on $[n]$ and that $\ker\widetilde{\Phi}=I_\Gamma$. First, $\Gamma$ is closed under taking subsets: if $A\in\Gamma$ and $B\subseteq A$, then $\mathbf{y}^{\mathbf{e}_A}=\mathbf{y}^{\mathbf{e}_B}\cdot\mathbf{y}^{\mathbf{e}_{A\setminus B}}\neq0$, so $\mathbf{y}^{\mathbf{e}_B}\neq0$ and $B\in\Gamma$. Since the empty product equals $1\neq0$, $\emptyset\in\Gamma$, and since $y_i\neq0$ for each $i$, $\Gamma$ contains all singletons. Second, since $\ker\widetilde{\Phi}$ is a monomial radical ideal, it is generated by squarefree monomials; a squarefree monomial $\mathbf{x}^{\mathbf{e}_S}$ belongs to $\ker\widetilde{\Phi}$ if and only if $\prod_{l\in S}y_l=0$, i.e., $S\notin\Gamma$, which is precisely the generating set of $I_\Gamma$. Thus $\widetilde{\Phi}$ induces an isomorphism $\phi\colon\KK[\Gamma]\xrightarrow{\ \sim\ }\KK[\Delta]$ with $\phi(\overline{x}_i)=y_i$.

We next show that $\Gamma$ is a relabeling of $\Delta$. The minimal primes of $\KK[\Delta]$ are the $\mathfrak{q}_F$, $F\in\mathcal{F}(\Delta)$, and, transported through $\phi$, also the ideals $(y_l : l\notin G)$, $G\in\mathcal{F}(\Gamma)$. Hence there is a bijection $\mathcal{F}(\Delta)\to\mathcal{F}(\Gamma)$, $F\mapsto G(F)$, with
$$
\mathfrak{q}_F=(y_l : l\notin G(F)).
$$
For a nonempty $\mathcal{S}\subseteq\mathcal{F}(\Delta)$, computing $\KK[\Delta]/\sum_{F\in\mathcal{S}}\mathfrak{q}_F$ with \eqref{eq:intersections} in the presentation by the $\overline{x}_i$ gives a polynomial ring of dimension $|\bigcap_{F\in\mathcal{S}}F|$, and through $\phi$ one of dimension $|\bigcap_{F\in\mathcal{S}}G(F)|$. Hence
$$
\Bigl|\bigcap_{F\in\mathcal{S}}F\Bigr|=\Bigl|\bigcap_{F\in\mathcal{S}}G(F)\Bigr|
\qquad\text{for every nonempty }\mathcal{S}\subseteq\mathcal{F}(\Delta).
$$
Both families of facets cover $[n]$, since all singletons are faces of $\Delta$ and of $\Gamma$. By \cref{lem:matching} there is $\sigma\in S_n$ with $\sigma(F)=G(F)$ for all $F$, whence $\mathcal{F}(\Gamma)=\{\sigma(F): F\in\mathcal{F}(\Delta)\}$ and $\Gamma=\sigma\cdot\Delta:=\{\sigma(A): A\in\Delta\}$.

The conjugacy follows. The permutation automorphism of $\KK[\mathbf{x}]$, $x_i\mapsto x_{\sigma(i)}$, maps $I_\Delta$ onto $I_{\sigma\cdot\Delta}=I_\Gamma$ and induces an isomorphism $\overline{p}_\sigma\colon\KK[\Delta]\to\KK[\Gamma]$. Set $\Psi:=\phi\circ\overline{p}_\sigma\in\Aut_\KK(\KK[\Delta])$, so that $\Psi(\overline{x}_i)=y_{\sigma(i)}$. For $d\in D$,
$$
(\Psi^{-1}\circ d\circ\Psi)(\overline{x}_i)=\Psi^{-1}\bigl(\chi_{\sigma(i)}(d)\,y_{\sigma(i)}\bigr)=\chi_{\sigma(i)}(d)\,\overline{x}_i,
$$
so $\Psi^{-1}D\Psi\subseteq T_0$. Define $\iota\colon D\to T_0$ by $\iota(d):=(\chi_{\sigma(1)}(d),\dots,\chi_{\sigma(n)}(d))$. Since each $\chi_{\sigma(i)}$ is a character of $D$, the map $\iota$ is a homomorphism of algebraic groups. Since $\sigma$ is a permutation, $\chi_{\sigma(1)},\dots,\chi_{\sigma(n)}$ is a $\ZZ$-basis of $X(D)$, so $\iota$ is an isomorphism $D\xrightarrow{\sim}(\KK^*)^n$ (see the proof of $(3)\Rightarrow(1)$ of proposition in \cite[Section~8.5]{Bor91}). Its image is a closed subtorus of $T_0$ of dimension $n$, hence all of $T_0$. Therefore $\Psi^{-1}D\Psi=T_0$.

\end{proof}

\begin{example}\label{ex:same-aut}
The fully exposed hypothesis is sharp in the strongest possible sense: for the complexes of \cref{ex:rigid-invisible}, not only the root sets but the full automorphism ind-groups agree. \rchange{\cref{fig:same-aut} records the component structure the proof below manipulates: each shaded wedge is an irreducible component $V_F\cong\AF^2$ and each ray a coordinate line $V_{\{i\}}$. $X_{\Delta_1}$ is a chain of four planes, consecutive ones meeting in a line; $X_{\Delta_2}$ is a chain of three planes together with a line meeting them only at the origin. The two are not isomorphic, yet the argument produces the same automorphisms for both: a torus, one polynomial parameter at each end of the chain, and the reversal $\tau$.}

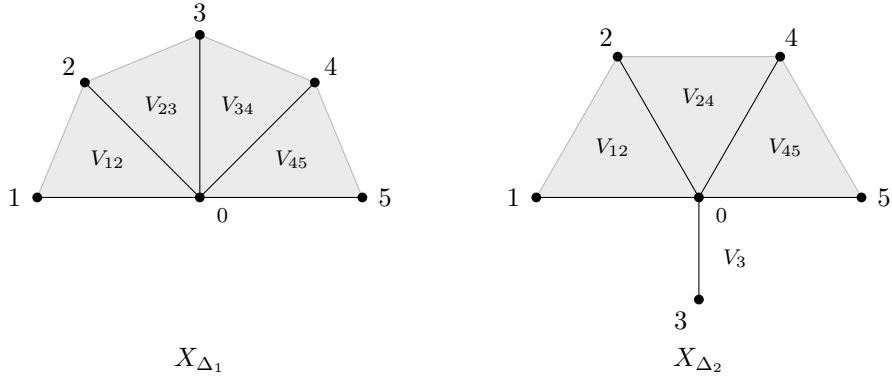
\begin{figure}[H]
\centering
\begin{tikzpicture}[scale=1.0, v/.style={circle,fill,inner sep=1.3pt}]
\begin{scope}
  \foreach \s/\t in {180/135,135/90,90/45,45/0}
     {\fill[gray!16] (0,0)--(\s:2.15)--(\t:2.15)--cycle;}
  \foreach \s/\t in {180/135,135/90,90/45,45/0}
     {\draw[gray!55] (\s:2.15)--(\t:2.15);}
  \foreach \a in {180,135,90,45,0} {\draw (0,0)--(\a:2.15);}
  \foreach \a/\l in {180/1,135/2,90/3,45/4,0/5}
     {\node[v] at (\a:2.15) {}; \node at (\a:2.45) {\small $\l$};}
  \node[v] at (0,0) {};
  \node at (0.30,-0.24) {\scriptsize $0$};
  \foreach \m/\l in {157.5/{12},112.5/{23},67.5/{34},22.5/{45}}
     {\node at (\m:1.32) {\scriptsize $V_{\l}$};}
  \node at (0,-2.15) {\small $X_{\Delta_1}$};
\end{scope}
\begin{scope}[xshift=6.6cm]
  \foreach \s/\t in {180/120,120/60,60/0}
     {\fill[gray!16] (0,0)--(\s:2.15)--(\t:2.15)--cycle;}
  \foreach \s/\t in {180/120,120/60,60/0}
     {\draw[gray!55] (\s:2.15)--(\t:2.15);}
  \foreach \a in {180,120,60,0} {\draw (0,0)--(\a:2.15);}
  \foreach \a/\l in {180/1,120/2,60/4,0/5}
     {\node[v] at (\a:2.15) {}; \node at (\a:2.45) {\small $\l$};}
  \draw (0,0)--(270:1.35);
  \node[v] at (270:1.35) {};
  \node at (-0.24,-1.62) {\small $3$};
  \node[v] at (0,0) {};
  \node at (0.30,-0.24) {\scriptsize $0$};
  \foreach \m/\l in {150/{12},90/{24},30/{45}}
     {\node at (\m:1.32) {\scriptsize $V_{\l}$};}
  \node at (0.48,-0.80) {\scriptsize $V_{3}$};
  \node at (0,-2.15) {\small $X_{\Delta_2}$};
\end{scope}
\end{tikzpicture}
\caption{The varieties $X_{\Delta_1}$ and $X_{\Delta_2}$.}
\label{fig:same-aut}
\end{figure}
\end{example}

\begin{proposition}\label{prop:same-aut}
Let $\Delta_1$ and $\Delta_2$ be the complexes of \cref{ex:rigid-invisible}. For parameters $t\in(\KK^{\times})^5$ and $h,g\in\KK[x]$, and for $\Delta\in\{\Delta_1,\Delta_2\}$, the formulas
$$
\varphi_{t,h,g}\colon\quad
\overline{x}_1\mapsto t_1\overline{x}_1,\quad
\overline{x}_2\mapsto t_2\overline{x}_2+\overline{x}_1h(\overline{x}_1),\quad
\overline{x}_3\mapsto t_3\overline{x}_3,\quad
\overline{x}_4\mapsto t_4\overline{x}_4+\overline{x}_5g(\overline{x}_5),\quad
\overline{x}_5\mapsto t_5\overline{x}_5
$$
define automorphisms of $\KK[\Delta]$, and, with $\tau:=\rho_{(1\,5)(2\,4)}$,
$$
\Aut(X_{\Delta_1})\;=\;\{\varphi_{t,h,g}\}\rtimes\langle\tau\rangle
\qquad\text{and}\qquad
\Aut(X_{\Delta_2})\;=\;\{\varphi_{t,h,g}\}\rtimes\langle\tau\rangle,
$$
given by the same formulas. In particular $\Aut(X_{\Delta_1})$ and $\Aut(X_{\Delta_2})$ are isomorphic as ind-groups, while $X_{\Delta_1}\not\cong X_{\Delta_2}$. Consequently, the fully exposed hypothesis cannot be removed from \cref{thm:main}.
\end{proposition}

\begin{proof}
We first verify that the formulas define automorphisms. The image of $\overline{x}_2$ is supported on the faces $\{2\}$ and $\{1\}$, and that of $\overline{x}_4$ on $\{4\}$ and $\{5\}$; consequently, the image of a minimal generator $\overline{x}_a\overline{x}_b$ of $I_{\Delta_r}$ is a combination of monomials supported on $\{a'\}\cup\{b'\}$ with $a'\in\pi(a)$ and $b'\in\pi(b)$, where $\pi(2)=\{1,2\}$, $\pi(4)=\{4,5\}$ and $\pi(c)=\{c\}$ otherwise. Running over the minimal non-faces of $\Delta_1$ (the pairs $\{1,3\},\{1,4\},\{1,5\},\{2,4\},\{2,5\},\{3,5\}$) and of $\Delta_2$ (the pairs $\{1,3\},\{1,4\},\{1,5\},\{2,3\},\{2,5\},\{3,4\},\{3,5\}$), all resulting pairs are again non-faces, so both ideals are preserved. The composition law
$$
\varphi_{t,h,g}\circ\varphi_{t',h',g'}=\varphi_{tt',\,h'',\,g''},\qquad
h''(x)=t_2'\,h(x)+t_1\,h'(t_1x),\quad
g''(x)=t_4'\,g(x)+t_5\,g'(t_5x),
$$
shows that these maps form a subgroup $G\cong(\KK[x]\oplus\KK[x])\rtimes T_0$ of $\Aut(X_{\Delta_r})$; and $\tau\in S(\Delta_r)$ for $r=1,2$ normalizes $G$, exchanging the parameters $h$ and $g$ and the coordinates of $t$ accordingly.

We now show that every automorphism is of this form. Fix $r$ and $\varphi\in\Aut(X_{\Delta_r})$. The components of $X_{\Delta_1}$ are the planes $V_{\{1,2\}},V_{\{2,3\}},V_{\{3,4\}},V_{\{4,5\}}$, consecutive ones meeting in a line and the remaining pairs only at the origin; those of $X_{\Delta_2}$ are the planes $V_{\{1,2\}},V_{\{2,4\}},V_{\{4,5\}}$ with the same pattern, together with the line $V_{\{3\}}$, which meets the planes only at the origin. In both cases $\varphi$ permutes the components preserving dimensions and the dimensions of pairwise intersections, hence induces an automorphism of a path with four, respectively three, nodes, fixing the line component when present; the induced permutation is therefore trivial or the reversal, and the reversal is realized by $\tau$. After composing with $\tau$ if necessary, $\varphi$ fixes every component, and hence fixes the lines $V_{\{2\}}$ and $V_{\{4\}}$ (and $V_{\{3\}}$), each being a pairwise intersection of components or a component.

On a middle plane ($V_{\{2,3\}}$ and $V_{\{3,4\}}$ for $\Delta_1$; $V_{\{2,4\}}$ for $\Delta_2$) both coordinate axes are among the fixed lines, so the restriction $\psi$ of $\varphi$ and its inverse preserve both coordinate ideals of the plane; writing $\psi(x_a)=x_a\,p$ and $\psi^{-1}(x_a)=x_a\,q$, we get $x_a=\psi(x_aq)=x_a\,p\,\psi(q)$, so $p\,\psi(q)=1$ and $p\in\KK[x_a,x_b]^{\times}=\KK^{\times}$: the restriction is diagonal. On the line $V_{\{3\}}$ of $X_{\Delta_2}$, $\varphi$ restricts to an automorphism of $\AF^1$ fixing the origin (\cref{lem:origin}), hence to a scaling. On the end plane $V_{\{1,2\}}\cong\Spec\KK[x_1,x_2]$ only the axis $V_{\{2\}}=\{x_1=0\}$ is forced to be fixed; the same unit argument applied to the ideal $(x_1)$ gives $\varphi(x_1)=t_1x_1$, and, since $\varphi$ induces a linear automorphism of the fixed line $V_{\{2\}}$, we may write $\varphi(x_2)=t_2x_2+x_1H(x_1,x_2)$. The Jacobian determinant of an automorphism of $\AF^2$ is a nonzero constant, and here it equals $t_1\bigl(t_2+x_1\,\partial H/\partial x_2\bigr)$, forcing $\partial H/\partial x_2=0$, i.e., $H=h(x_1)$. The end plane $V_{\{4,5\}}$ is treated symmetrically, yielding $\varphi(x_5)=t_5x_5$ and $\varphi(x_4)=t_4x_4+x_5g(x_5)$. Finally, the scalars attached to a vertex agree across the components containing it (restrict to the common line), and comparing restrictions to all components, using $\bigcap_F\mathfrak{q}_F=0$, identifies $\varphi$ with $\varphi_{t,h,g}$ globally.

The description of $\Aut(X_{\Delta_r})$ by the parameters $(t,h,g,\tau^{\epsilon})$, together with the composition law above and the action of $\tau$, is literally the same for $r=1$ and $r=2$; matching parameters therefore defines a group isomorphism $\Theta\colon\Aut(X_{\Delta_1})\to\Aut(X_{\Delta_2})$. That $\Theta$ is an isomorphism of ind-groups is now a matter of reading off the filtrations. Recall from \cref{subsec:ind} that the $d$-th filtration piece of $\Aut(X_{\Delta_r})$ consists of the pairs $(\varphi,\varphi^{-1})$ with $\deg\varphi(\overline{x}_l)\leq d$ and $\deg\varphi^{-1}(\overline{x}_l)\leq d$ for all $l$. For $\varphi=\varphi_{t,h,g}\tau^{\epsilon}$ the first condition reads $\deg h\leq d-1$ and $\deg g\leq d-1$; and the composition law shows that $\varphi_{t,h,g}^{-1}=\varphi_{t^{-1},\tilde h,\tilde g}$ with $\tilde h(x)=-t_1^{-1}t_2^{-1}\,h(t_1^{-1}x)$ and $\tilde g(x)=-t_4^{-1}t_5^{-1}\,g(t_5^{-1}x)$, so $\deg\tilde h=\deg h$ and $\deg\tilde g=\deg g$ and the second condition is implied by the first. Hence for every $d\geq1$ the $d$-th piece is, for both $r=1$ and $r=2$,
$$
\Aut(X_{\Delta_r})_{\leq d}\;=\;\bigl\{(t,h,g,\epsilon)\ :\ \deg h\leq d-1,\ \deg g\leq d-1\bigr\}\;\cong\;(\KK^{\times})^{5}\times\AF^{d}\times\AF^{d}\times\ZZ/2,
$$
the coordinates on the right being $t$, the coefficient vectors of $h$ and $g$, and $\epsilon$. In these coordinates $\Theta$ is the identity, hence an isomorphism of varieties on each piece, and therefore an isomorphism of ind-groups. On the other hand, the multiset of dimensions of the irreducible components is $(2,2,2,2)$ for $X_{\Delta_1}$ and $(2,2,2,1)$ for $X_{\Delta_2}$, so the varieties are not isomorphic.
\end{proof}

Multiplying by an affine line repairs the failure, here and in general; see
\cref{cor:stabilized} and \cref{cor:product-line} in \cref{sec:main}.

At the opposite extreme, the rigid case admits a closed-form classification. Recall from \cref{prop:rigid-cl} and \cref{cor:hidden-rigid} that $\Delta$ is fully hidden precisely when every face of $\Delta$ is an intersection of facets, and precisely when $X_\Delta$ is rigid; the proof below uses the first description.

\begin{theorem}\label{thm:rigid-aut}
Let $\Delta$ be a fully hidden simplicial complex on $[n]$. Then
$$
\Aut(X_\Delta)\ =\ T_0\ \rtimes\ S(\Delta),
$$
a linear algebraic group of dimension~$n$ whose identity component is $T_0$.\rchange{\footnote{%
For \emph{irreducible} rigid affine varieties, \cite[Theorem~1]{AG17} produces a subtorus of
$\Aut(X)$ containing every other subtorus, which is therefore normal. That statement does not
cover the present case: a fully hidden $\Delta$ on $n\geq2$ vertices has at least two facets;
a single facet would be all of $[n]$, every vertex of which is private, so $X_\Delta$ is
reducible. Their Proposition~1 does give a finite extension of the torus when the canonical
grading is pointed, a hypothesis $\KK[\Delta]$ satisfies (its weight cone is $\ZZ^n_{\geq0}$
and its degree-zero part is $\KK$), but only for an integral graded algebra, and the extension
is not asserted to split. Whether $\Aut(X)$ is a finite extension of a torus for a general rigid
$X$ is open even under strong hypotheses \cite[Remark~6.5]{BGa25}. The proof below is self-contained and gives more: the
extension splits, with finite complement $S(\Delta)$.}}
\end{theorem}
\begin{proof}
Let $\varphi\in\Aut(X_\Delta)$. By \cref{lem:origin} it permutes the components $\{V_F\}$; write $F\mapsto G(F)$ for the induced permutation of $\mathcal{F}(\Delta)$, so $\varphi(V_F)=V_{G(F)}$ and $\varphi^*(\mathfrak{q}_{G(F)})=\mathfrak{q}_F$, the inclusion of the proof of \cref{lem:components} being an equality because $\varphi^*$ carries minimal primes to minimal primes. Hence for every nonempty $\mathcal{S}\subseteq\mathcal{F}(\Delta)$ the map $\varphi^*$ carries $\sum_{F\in\mathcal{S}}\mathfrak{q}_{G(F)}$ onto $\sum_{F\in\mathcal{S}}\mathfrak{q}_F$ and induces an isomorphism of the quotients, which by \eqref{eq:intersections} are polynomial rings of dimensions $|\bigcap_{F\in\mathcal{S}}G(F)|$ and $|\bigcap_{F\in\mathcal{S}}F|$; so
$$
\Bigl|\bigcap_{F\in\mathcal{S}}F\Bigr|=\Bigl|\bigcap_{F\in\mathcal{S}}G(F)\Bigr|
\qquad\text{for every nonempty }\mathcal{S}\subseteq\mathcal{F}(\Delta),
$$
exactly as in the proof of \cref{thm:tori}\ref{tori.c}. Both facet families cover $[n]$ (every singleton is a face), so \cref{lem:matching} produces $\sigma\in S_n$ with $\sigma(F)=G(F)$ for every facet $F$; in particular $\sigma\in S(\Delta)$. Composing with $\rho_{\sigma^{-1}}$, we reduce to $\varphi(V_F)=V_F$ for every $F$.

By \cref{prop:rigid-cl}\ref{rc.c}, every face $S\in\Delta$ is an intersection of facets, so the coordinate subspace $V_S=\bigcap_{F\in\mathcal{F}_S(\Delta)}V_F$ is a finite intersection of $\varphi$-stable components and is $\varphi$-stable. In particular, on any component $V_F=\Spec\KK[x_j:j\in F]$, each coordinate ideal $(x_j)$ (with $j\in F$) is $\varphi|_{V_F}$-stable, since it cuts out the $\varphi$-stable subspace $V_{F\setminus\{j\}}\subseteq V_F$ (a face, by \cref{prop:rigid-cl}\ref{rc.c}). Write $\psi:=\varphi|_{V_F}$; then $\psi(x_j)=x_j\,p_j$ for some $p_j\in\KK[F]$. Applying the same to $\psi^{-1}$ gives $\psi^{-1}(x_j)=x_j\,q_j$, and
$$
x_j=\psi(\psi^{-1}(x_j))=\psi(x_j q_j)=(x_j p_j)\,\psi(q_j)=x_j\,p_j\psi(q_j),
$$
so $p_j\psi(q_j)=1$ in $\KK[F]$, forcing $p_j\in\KK[F]^\times=\KK^\times$. Thus $\psi$ is the diagonal automorphism $x_j\mapsto t_j^{(F)}x_j$ with scalars $t_j^{(F)}\in\KK^\times$, $j\in F$.

The scalars are independent of the facet chosen. If $F,G\in\mathcal{F}(\Delta)$ and $j\in F\cap G$, the line $V_{\{j\}}\subseteq V_F\cap V_G$ is $\varphi$-stable (again by \cref{prop:rigid-cl}\ref{rc.c}), and the restrictions of $\varphi|_{V_F}$ and $\varphi|_{V_G}$ to $V_{\{j\}}$ both equal multiplication by $t_j^{(F)}$ and $t_j^{(G)}$ respectively, hence $t_j^{(F)}=t_j^{(G)}=:t_j$. Setting $t:=(t_1,\dots,t_n)\in T_0$, the automorphism $\varphi$ and the diagonal $\varphi_t\in T_0$ agree on every component, hence agree modulo $\bigcap_F\mathfrak{q}_F=0$, i.e., $\varphi=\varphi_t$.

Every $\varphi\in\Aut(X_\Delta)$ therefore factors uniquely as $\varphi_t\circ\rho_\sigma$ with $t\in T_0$ and $\sigma\in S(\Delta)$; conversely, every such product is an automorphism. The action of $S(\Delta)$ on $T_0$ is the standard permutation of coordinates, giving $\Aut(X_\Delta)=T_0\rtimes S(\Delta)$. Since $T_0$ is a closed algebraic subgroup of dimension $n$ (\cref{thm:tori}\ref{tori.a}) and $S(\Delta)$ is finite, the whole is a linear algebraic group of dimension $n$. Finally, $T_0$ is connected and of finite index $|S(\Delta)|$ in $T_0\rtimes S(\Delta)$, hence is open and closed in it; a connected open subgroup containing the identity is the identity component, so $\Aut(X_\Delta)^\circ=T_0$.
\end{proof}

\begin{remark}\label{rem:rigid-vs-nonrigid}
A rigid Stanley-Reisner variety never shares its automorphism ind-group with a non-rigid one,
fully exposed ones included. Indeed, if $\mathcal{R}(\Delta')\neq\emptyset$ then
$\Aut(X_{\Delta'})$ contains a root subgroup, hence a copy of $\GA$
(\cref{prop:root-subgroups}); whereas for fully hidden $\Delta$ the group
$\Aut(X_\Delta)=T_0\rtimes S(\Delta)$ contains none, any morphism $\GA\to T_0\rtimes S(\Delta)$
factoring through $T_0$ and every algebraic-group morphism $\GA\to T_0$ being trivial. Since an
isomorphism of ind-groups carries algebraic $\GA$-subgroups to algebraic $\GA$-subgroups
(\cref{lem:image-algebraic}), the two ind-groups are not isomorphic. In particular the description
of \cref{thm:rigid-aut} characterizes the fully hidden complexes: if
$\Aut(X_\Delta)\cong T_0\rtimes S(\Delta)$ as ind-groups then $\Delta$ is fully hidden, since
otherwise $\mathcal{R}(\Delta)\neq\emptyset$ and the left-hand side would contain a $\GA$ that the
right-hand side does not. Consequently the fully
exposed hypothesis of \cref{thm:main} cannot be removed by passing to a model: no assignment
$\Delta\mapsto\Phi(\Delta)$ producing a fully exposed complex with
$\Aut(X_{\Phi(\Delta)})\cong\Aut(X_\Delta)$ can be defined on the rigid stratum.
\end{remark}

\section{The characterization theorem}\label{sec:main}

We fix two fully exposed simplicial complexes $\Delta$ on $[n]$ and $\Delta'$ on $[n']$; recall that every singleton is a face (\cref{subsec:sr}) and that full exposedness forces $\mathcal{R}(\Delta)\neq\emptyset$. An isomorphism of ind-groups $\theta\colon\Aut(X_\Delta)\to\Aut(X_{\Delta'})$ transports the maximal torus and the root subgroups, hence induces $w\in\GL_n(\ZZ)$ with $w(\mathcal{R}(\Delta))=\mathcal{R}(\Delta')$ (\cref{thm:transfer,prop:linear}). The work is to show that such a $w$ must be a permutation matrix (\cref{thm:permutation}); \cref{thm:core} then converts this into the Main Theorem.

\begin{proposition}\label{thm:transfer}
Let $\theta\colon\Aut(X_\Delta)\to\Aut(X_{\Delta'})$ be an isomorphism of ind-groups. Then:
\begin{enumerate}[label=\textup{(\alph*)}]
\item\label{tr.a} $\theta(T_0)$ is an algebraic subgroup of $\Aut(X_{\Delta'})$ and $\theta|_{T_0}\colon T_0\to\theta(T_0)$ is an isomorphism of algebraic groups; in particular $\theta(T_0)$ is an algebraic torus of dimension $n$.
\item\label{tr.b} $n=n'$, and after composing $\theta$ with conjugation by some $g\in\Aut(X_{\Delta'})$ we may assume $\theta(T_0)=T_0'$.
\item\label{tr.c} Assume moreover that $\theta(T_0)=T_0'$, as arranged in \ref{tr.b}. Then for every root subgroup $U\subseteq\Aut(X_\Delta)$ with respect to $T_0$, of weight $\alpha$, the image $\theta(U)$ is a root subgroup of $\Aut(X_{\Delta'})$ with respect to $T_0'$, of weight $\alpha\circ(\theta|_{T_0})^{-1}$.
\end{enumerate}
\end{proposition}

\begin{proof}
\ref{tr.a} is \cref{lem:image-algebraic} applied to $A=T_0$.

\ref{tr.b} By \ref{tr.a}, $\Aut(X_{\Delta'})$ contains a torus of dimension $n$, so $n\leq n'$ by \cref{thm:tori} applied to $\Delta'$; applying the same argument to $\theta^{-1}$ gives $n'\leq n$, hence $n=n'$. Any torus $T'\supseteq\theta(T_0)$ satisfies $\dim T'\leq n'=\dim\theta(T_0)$, and a closed connected subgroup of full dimension in a torus is the whole torus; hence $\theta(T_0)$ is a maximal torus; since $\dim\theta(T_0)=n'$, \cref{thm:tori} yields $g\in\Aut(X_{\Delta'})$ with $g\,\theta(T_0)\,g^{-1}=T_0'$. Replace $\theta$ by $t\mapsto g\,\theta(t)\,g^{-1}$.

\ref{tr.c} Assume $\theta(T_0)=T_0'$ and let $(U,\varepsilon)$ be a root subgroup of weight $\alpha$. By \cref{lem:image-algebraic}, $\theta(U)$ is an algebraic subgroup isomorphic to $\GA$, with parametrization $\varepsilon':=\theta\circ\varepsilon$. Every $t'\in T_0'$ is of the form $t'=\theta(t)$ with $t\in T_0$, and
$$
t'\,\varepsilon'(s)\,t'^{-1}=\theta\bigl(t\,\varepsilon(s)\,t^{-1}\bigr)=\theta\bigl(\varepsilon(t^\alpha s)\bigr)=\varepsilon'\bigl(t^\alpha s\bigr)
=\varepsilon'\Bigl(\bigl(\alpha\circ(\theta|_{T_0})^{-1}\bigr)(t')\cdot s\Bigr),
$$
and the character $\alpha\circ(\theta|_{T_0})^{-1}$ of $T_0'$ is nontrivial. Hence $\theta(U)$ is a root subgroup of the stated weight.
\end{proof}

\begin{proposition}\label{prop:linear}
Let $\theta\colon\Aut(X_\Delta)\to\Aut(X_{\Delta'})$ be an isomorphism of ind-groups with $\theta(T_0)=T_0'$, and let $\vartheta:=\theta|_{T_0}\colon T_0\xrightarrow{\sim} T_0'$. Let $w:=(\vartheta^{-1})^*\colon X(T_0)\to X(T_0')$ be the induced isomorphism of character lattices, regarded as an element of $\GL_n(\ZZ)$ after the identifications $X(T_0)=X(T_0')=\ZZ^n$. Then $\theta$ maps the root subgroup of weight $\alpha$ to a root subgroup of weight $w\cdot\alpha$, and
$$
w\bigl(\mathcal{R}(\Delta)\bigr)=\mathcal{R}(\Delta').
$$
\end{proposition}

\begin{proof}
By \cref{thm:transfer}\ref{tr.c}, $\theta$ maps every root subgroup of weight $\alpha$ to a root subgroup of weight $\alpha\circ\vartheta^{-1}$. Since $w=(\vartheta^{-1})^*$, this weight equals $w\cdot\alpha$. By \cref{prop:root-subgroups}, the weights of root subgroups of $\Aut(X_\Delta)$ with respect to $T_0$ are exactly $\mathcal{R}(\Delta)$, and those of $\Aut(X_{\Delta'})$ with respect to $T_0'$ are exactly $\mathcal{R}(\Delta')$; hence $w(\mathcal{R}(\Delta))\subseteq\mathcal{R}(\Delta')$. Applying the same argument to $\theta^{-1}$ gives $w^{-1}(\mathcal{R}(\Delta'))\subseteq\mathcal{R}(\Delta)$, i.e., $\mathcal{R}(\Delta')\subseteq w(\mathcal{R}(\Delta))$. Therefore $w(\mathcal{R}(\Delta))=\mathcal{R}(\Delta')$.
\end{proof}

It remains to show that $w$ is a permutation matrix. We partition the vertex set $[n]$ according to the roles played by each index in the root set.

\begin{definition}\label{def:support-target}
The sets of \emph{support vertices} and of \emph{target-only vertices} of $\Delta$ are
$$
\Supp(\Delta):=\bigcup_{i\in[n]}\ \bigcup_{S\in\mathcal{A}_i(\Delta)} S,
\qquad
\mathcal{O}(\Delta):=[n]\setminus\Supp(\Delta).
$$
We write $\ZZ^{\Supp(\Delta)}:=\bigoplus_{j\in\Supp(\Delta)}\ZZ\,\mathbf{e}_j\subseteq\ZZ^n$.
\end{definition}

\begin{lemma}\label{lem:target-properties}
Let $\Delta$ be fully exposed, and let $j\in\mathcal{O}(\Delta)$. Then $\mathcal{A}_j(\Delta)\neq\emptyset$ and $\bigcap_{S\in\mathcal{A}_j(\Delta)}S=\emptyset$.
\end{lemma}

\begin{proof}
If $\Delta$ has a cone vertex $i_0$, then for every $l\notin c(\Delta)$ there is a facet $G\ni l$ with $i_0\in G$ and $l\neq i_0$; by \cref{lem:cone-roots}\ref{cr.1}, $G\setminus\{i_0\}\in\mathcal{A}_{i_0}(\Delta)$, so $l\in\Supp(\Delta)$. Hence $\mathcal{O}(\Delta)\subseteq c(\Delta)$, and for $j\in\mathcal{O}(\Delta)$ we have $\emptyset\in\mathcal{A}_j(\Delta)$ by \cref{lem:cone-roots}, giving $\mathcal{A}_j(\Delta)\neq\emptyset$ and $\bigcap_{S\in\mathcal{A}_j(\Delta)}S=\emptyset$.

Assume now that $\Delta$ has no cone vertices. Let $G$ be a facet with $j\in G$. Since $\Delta$ is fully exposed, $G$ has a private vertex $p$ by \cref{lem:private-vertex}, and $G\setminus\{p\}\in\mathcal{A}_p(\Delta)$ is nonempty, so every vertex of $G\setminus\{p\}$ lies in $\Supp(\Delta)$. Since $j\in\mathcal{O}(\Delta)$, we must have $p=j$. Thus $j$ is a private vertex of every facet containing it, whence $G\setminus\{j\}\in\mathcal{A}_j(\Delta)$ for every facet $G\ni j$, giving $\mathcal{A}_j(\Delta)\neq\emptyset$.

For the intersection, suppose $c\in\bigcap_{S\in\mathcal{A}_j(\Delta)}S$; note $c\neq j$. Then $c$ belongs to every facet containing $j$: if $G\ni j$ is a facet, then $G\setminus\{j\}\in\mathcal{A}_j(\Delta)$, so $c\in G\setminus\{j\}\subseteq G$. Fix a facet $G_0\ni j$. If $c$ were a private vertex of $G_0$, then $G_0\setminus\{c\}\in\mathcal{A}_c(\Delta)$ and $j\in G_0\setminus\{c\}\subseteq\Supp(\Delta)$, contradicting $j\in\mathcal{O}(\Delta)$. Hence $c$ is not private in $G_0$, so there is a facet $G'\neq G_0$ with $G'\supseteq G_0\setminus\{c\}$. Then $j\in G_0\setminus\{c\}\subseteq G'$, so $c\in G'$, whence $G'\supseteq G_0$, contradicting $G'\neq G_0$.
\end{proof}

A vector $d\in\mathbb Z^n\setminus{0}$ is called primitive if $\gcd(d_1,\ldots,d_n)=1$.
\begin{lemma}\label{lem:rays}
Let $\Delta$ satisfy $\mathcal{R}(\Delta)\neq\emptyset$, and define the set of \emph{ray directions}
$$
\mathcal{D}(\Delta):=\bigl\{d\in\ZZ^n \text{ primitive}\;:\; \alpha+\ZZ_{\geq0}\,d\subseteq\mathcal{R}(\Delta) \text{ for some } \alpha\in\mathcal{R}(\Delta)\bigr\}.
$$
Then
$$
\mathcal{D}(\Delta)=\bigl\{d\in\ZZ^n_{\geq0}\text{ primitive}\;:\;\emptyset\neq\supp(d)\subseteq S \text{ for some } S\in\textstyle\bigcup_i\mathcal{A}_i(\Delta)\bigr\},
$$
and $\cone(\mathcal{D}(\Delta))=\RR_{\geq0}^{\Supp(\Delta)}:=\cone(\mathbf{e}_j : j\in\Supp(\Delta))$, a simplicial cone whose extreme rays are the $\RR_{\geq0}\mathbf{e}_j$ with $j\in\Supp(\Delta)$.
\end{lemma}

\begin{proof}
For the inclusion $\supseteq$, let $d\in \Z^n_{\geq 0}$ be primitive with $\emptyset\neq\supp(d)\subseteq S\in\mathcal{A}_i(\Delta)$. Then $i\notin S$ gives $d_i=0$, and for $\alpha:=\mathbf{e}_S-\mathbf{e}_i$ and every $k\geq0$ we have $\alpha+kd=(\mathbf{e}_S+kd)-\mathbf{e}_i$ with $\mathbf{e}_S+kd\in\ZZ^n_{\geq0}$, $(\mathbf{e}_S+kd)_i=0$ and $\supp(\mathbf{e}_S+kd)=S\in\mathcal{A}_i(\Delta)$; hence $\alpha+kd\in\mathcal{R}(\Delta)$.

For the reverse inclusion, suppose $\alpha+kd\in\mathcal{R}(\Delta)$ for all $k\geq0$, where $\alpha=\beta-\mathbf{e}_i\in\mathcal{R}_i(\Delta)$. All coordinates of roots are $\geq-1$ (\cref{lem:root-data}), so a negative coordinate of $d$ would drive a coordinate of $\alpha+kd$ to $-\infty$; hence $d\in \Z^n_{\geq 0}$, and $d\neq0$ by primitivity. If $d_i>0$, then for $k\geq1$ the $i$-th coordinate $-1+kd_i$ is nonnegative, while $(\alpha+kd)_l=\beta_l+kd_l\geq0$ for $l\neq i$; then $\alpha+kd$ has no negative coordinate, which is impossible for a root. Hence $d_i=0$, and for $k=1$ we get $\alpha+d=(\beta+d)-\mathbf{e}_i\in\mathcal{R}_i(\Delta)$, so $S:=\supp(\beta+d)\in\mathcal{A}_i(\Delta)$ contains $\supp(d)$, which is nonempty.

For the last assertion: every $d\in\mathcal{D}(\Delta)$ has $\supp(d)\subseteq\Supp(\Delta)$, and conversely $\mathbf{e}_j\in\mathcal{D}(\Delta)$ for every $j\in\Supp(\Delta)$; hence $\cone(\mathcal{D}(\Delta))=\RR_{\geq0}^{\Supp(\Delta)}$, whose extreme rays are as stated.
\end{proof}

\begin{proposition}\label{thm:permutation}
Let $\Delta$ and $\Delta'$ be fully exposed simplicial complexes on $[n]$, and let $w\in\GL_n(\ZZ)$ satisfy $w(\mathcal{R}(\Delta))=\mathcal{R}(\Delta')$. Then $w$ is a permutation matrix.
\end{proposition}

\begin{proof}
Since $w$ is a bijection $\ZZ^n\to\ZZ^n$ mapping $\mathcal{R}(\Delta)$ onto $\mathcal{R}(\Delta')$ and preserving primitivity, it maps $\mathcal{D}(\Delta)$ bijectively onto $\mathcal{D}(\Delta')$, hence sends $\cone(\mathcal{D}(\Delta))$ onto $\cone(\mathcal{D}(\Delta'))$. By \cref{lem:rays}, $w$ maps $\RR_{\geq0}^{\Supp(\Delta)}$ onto $\RR_{\geq0}^{\Supp(\Delta')}$, taking extreme rays to extreme rays. Since $w$ preserves primitivity and the primitive generators of the extreme rays are standard basis vectors, there is a bijection $\pi\colon\Supp(\Delta)\to\Supp(\Delta')$ with
$$
w(\mathbf{e}_j)=\mathbf{e}_{\pi(j)}\qquad\text{for all } j\in\Supp(\Delta),
$$
and consequently $w(\ZZ^{\Supp(\Delta)})=\ZZ^{\Supp(\Delta')}$. For $\beta\in\ZZ^n_{\geq0}$ supported on $\Supp(\Delta)$ we write $\pi_*(\beta):=w(\beta)=\sum_l\beta_l\,\mathbf{e}_{\pi(l)}$.

We now determine the action of $w$ on target-only vertices. Since $|\Supp(\Delta)|=|\Supp(\Delta')|$, hence $|\mathcal{O}(\Delta)|=|\mathcal{O}(\Delta')|$, and $w$ induces an isomorphism
$$
C\colon\ZZ^n/\ZZ^{\Supp(\Delta)}\longrightarrow\ZZ^n/\ZZ^{\Supp(\Delta')},
$$
where the classes $\overline{\mathbf{e}}_j$, $j\in\mathcal{O}(\Delta)$ (resp.\ $j'\in\mathcal{O}(\Delta')$), form bases of the two sides. Fix $j\in\mathcal{O}(\Delta)$ and $S\in\mathcal{A}_j(\Delta)$, which is possible by \cref{lem:target-properties}; here $S=\emptyset$ is allowed, in which case $\mathbf{e}_S=0$. The image $w(\mathbf{e}_S-\mathbf{e}_j)\in\mathcal{R}(\Delta')$ has the form $\beta'-\mathbf{e}_{i'}$ with $\supp(\beta')\in\mathcal{A}_{i'}(\Delta')$, hence $\supp(\beta')\subseteq\Supp(\Delta')$. Reducing modulo $\ZZ^{\Supp(\Delta')}$ and using $w(\mathbf{e}_S)=\pi_*(\mathbf{e}_S)\in\ZZ^{\Supp(\Delta')}$ yields
$$
C(\overline{\mathbf{e}}_j)=\overline{\mathbf{e}}_{i'} .
$$
Since $C$ is an isomorphism and $\overline{\mathbf{e}}_j\neq0$, we must have $i'\in\mathcal{O}(\Delta')$; and $i'$ is determined by $C(\overline{\mathbf{e}}_j)$ alone, hence independent of the choice of $S$. Writing $i'=\tau(j)$, the map $\tau\colon\mathcal{O}(\Delta)\to\mathcal{O}(\Delta')$ satisfies $C(\overline{\mathbf{e}}_j)=\overline{\mathbf{e}}_{\tau(j)}$ and is therefore a bijection. Define $b_j\in\ZZ^{\Supp(\Delta')}$ by
$$
w(\mathbf{e}_j)=\mathbf{e}_{\tau(j)}+b_j,\qquad j\in\mathcal{O}(\Delta).
$$

We show that all entries of $b_j$ are nonpositive. Let $j\in\mathcal{O}(\Delta)$, $S\in\mathcal{A}_j(\Delta)$ and $\beta\in\ZZ^n_{\geq0}$ with $\supp(\beta)=S$ (so $\beta=0$ when $S=\emptyset$). Then $\beta-\mathbf{e}_j\in\mathcal{R}(\Delta)$ and
$$
w(\beta-\mathbf{e}_j)=\pi_*(\beta)-b_j-\mathbf{e}_{\tau(j)}\in\mathcal{R}(\Delta').
$$
The $\tau(j)$-th coordinate of this vector equals $-1$, since $\pi_*(\beta)$ and $b_j$ are supported on $\Supp(\Delta')$; a root has exactly one negative coordinate, located at its target (\cref{lem:root-data}), so $\tau(j)$ is the target and all remaining coordinates are nonnegative:
$$
(\pi_*(\beta))_l\ \geq\ (b_j)_l\qquad\text{for every } l\in[n].
$$
Taking $\beta=\mathbf{e}_S$ gives $b_j\leq\mathbf{e}_{\pi(S)}$: every entry of $b_j$ is at most $1$, with equality possible only at coordinates in $\pi(S)$. As $S\in\mathcal{A}_j(\Delta)$ was arbitrary, the set of coordinates where $b_j$ equals $1$ is contained in
$$
\bigcap_{S\in\mathcal{A}_j(\Delta)}\pi(S)=\pi\Bigl(\bigcap_{S\in\mathcal{A}_j(\Delta)}S\Bigr)=\emptyset
$$
by \cref{lem:target-properties}. Hence all entries of $b_j$ are $\leq0$.

It remains to show that $b_j=0\in\ZZ^n$ for every $j\in\mathcal{O}(\Delta)$. The matrix $w^{-1}$ satisfies the hypotheses with the roles of $\Delta$ and $\Delta'$ exchanged. By the preceding arguments, its support bijection is $\pi^{-1}$ and its target-only bijection is $\tau^{-1}$; writing $w^{-1}(\mathbf{e}_{\tau(j)})=\mathbf{e}_j+b^*_j$ with $b^*_j\in\ZZ^{\Supp(\Delta)}$ and applying $w$, we obtain $b_j=-\pi_*(b^*_j)$. By the same argument applied to $w^{-1}$ at the target-only vertex $\tau(j)\in\mathcal{O}(\Delta')$, all entries of $b^*_j$ are $\leq0$, hence all entries of $b_j=-\pi_*(b^*_j)$ are $\geq0$. Combined with the nonpositivity established above, $b_j=0$ for every $j\in\mathcal{O}(\Delta)$.

Therefore $w(\mathbf{e}_j)=\mathbf{e}_{\sigma(j)}$ for every $j\in[n]$, where $\sigma\in S_n$ is the permutation defined by $\sigma(j):=\pi(j)$ for $j\in\Supp(\Delta)$ and $\sigma(j):=\tau(j)$ for $j\in\mathcal{O}(\Delta)$. In particular, $w$ is a permutation matrix.
\end{proof}

\begin{example}\label{ex:full-needed}
The full root set is essential in \cref{thm:permutation}. Let $\Delta$ be the complex on $[4]$ with facets $\{1,2\}$ and $\{3,4\}$, so that $I_\Delta=(x_1x_3,\,x_1x_4,\,x_2x_3,\,x_2x_4)$: it is fully exposed (every vertex is a private vertex of its facet) and has no cone vertex, and
$$
\mathcal{R}(\Delta)=\{k\mathbf{e}_1-\mathbf{e}_2,\ k\mathbf{e}_2-\mathbf{e}_1,\ k\mathbf{e}_3-\mathbf{e}_4,\ k\mathbf{e}_4-\mathbf{e}_3 \,:\, k\geq1\}.
$$
The matrix
$$
w=\begin{pmatrix} 2&1&0&0\\ -1&0&0&0\\ 0&0&1&0\\ 0&0&0&1\end{pmatrix}\in\GL_4(\ZZ)
$$
preserves the square-free roots: $w(\mathbf{e}_1-\mathbf{e}_2)=\mathbf{e}_1-\mathbf{e}_2$, $w(\mathbf{e}_2-\mathbf{e}_1)=\mathbf{e}_2-\mathbf{e}_1$, and $w$ fixes $\pm(\mathbf{e}_3-\mathbf{e}_4)$. Yet $w(2\mathbf{e}_1-\mathbf{e}_2)=3\mathbf{e}_1-2\mathbf{e}_2\notin\mathcal{R}(\Delta)$, so $w$ does not preserve $\mathcal{R}(\Delta)$, consistently with \cref{thm:permutation}. Truncating the root set to its square-free part would break the reconstruction of the permutation for this (disconnected) complex; the asymptotic directions inside $\mathcal{R}(\Delta)$, captured by \cref{lem:rays}, carry the missing information. \rchange{The complex is drawn in \cref{fig:disjoint}: every vertex is private in its own facet, so $\Delta$ is fully exposed; but the two components can be rescaled independently, and that is what allows a non-permutation matrix to preserve the square-free roots while destroying $\mathcal{R}(\Delta)$.}
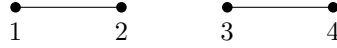
\begin{figure}[H]
\centering
\begin{tikzpicture}[scale=1.0, v/.style={circle,fill,inner sep=1.4pt}]
\draw (0,0)--(1.4,0); \draw (2.8,0)--(4.2,0);
\node[v,label=below:{\small $1$}] at (0,0) {};
\node[v,label=below:{\small $2$}] at (1.4,0) {};
\node[v,label=below:{\small $3$}] at (2.8,0) {};
\node[v,label=below:{\small $4$}] at (4.2,0) {};
\end{tikzpicture}
\caption{The disconnected fully exposed complex.}
\label{fig:disjoint}
\end{figure}
\end{example}

\begin{theorem}\label{thm:main}
Let $\Delta$ and $\Delta'$ be fully exposed simplicial complexes on $[n]$ and $[n']$. If there is an isomorphism of ind-groups
$$
\theta\colon\Aut(X_\Delta)\longrightarrow\Aut(X_{\Delta'}),
$$
then $n=n'$ and there is a permutation $\sigma\in S_n$ with $\sigma\cdot\Delta=\Delta'$. In particular $\Delta\cong\Delta'$ as abstract simplicial complexes and $X_\Delta\cong X_{\Delta'}$ as algebraic varieties.
\end{theorem}

\begin{proof}
By \cref{thm:transfer}, $n=n'$ and, after composing $\theta$ with an inner automorphism, $\theta(T_0)=T_0'$. By \cref{prop:linear}, there is $w\in\GL_n(\ZZ)$ with $w(\mathcal{R}(\Delta))=\mathcal{R}(\Delta')$. By \cref{thm:permutation}, $w$ is the permutation matrix of some $\sigma\in S_n$, so that $\sigma\cdot\mathcal{R}(\Delta)=\mathcal{R}(\Delta')$; by \cref{cor:core-permuted}, $\sigma\cdot\Delta=\Delta'$.

Finally, the permutation automorphism $x_i\mapsto x_{\sigma(i)}$ of $\KK[\mathbf{x}]$ maps $I_\Delta$ onto $I_{\sigma\cdot\Delta}=I_{\Delta'}$ and induces an isomorphism $\KK[\Delta]\cong\KK[\Delta']$, i.e., $X_\Delta\cong X_{\Delta'}$.
\end{proof}

\begin{remark}[Relation to M\"uller's theorem]\label{rem:muller-comparison}
\rchange{The closest precursor is \cite[Theorem~1]{Mul95}: two \emph{$T_1$-complexes} are
isomorphic if and only if the Lie algebras $\Der(\KK[\Delta])$ and $\Der(\KK[\Delta'])$ are, where
$\Delta$ is a $T_1$-complex when it has at least two vertices and, for every ordered pair of
vertices, some facet contains the first but not the second \cite[p.~135]{Mul95}; in the language
of \cref{subsec:closure}, when $\operatorname{cl}(\{j\})=\{j\}$ for every vertex $j$. The two
settings are closely related: the \emph{negative} roots of \cite[Proposition~3]{Mul95} are exactly
the elements of $\mathcal{R}(\Delta)$ (\cref{rem:muller-dictionary}), and the classes of complexes
do meet: the complex on $[6]$ with facets $\{1,2,3\},\{1,4,5\},\{2,4,6\},\{3,5,6\}$ is both
$T_1$ and fully exposed (\cref{fig:overlap}). The hypotheses are nevertheless independent, as two
examples show: no cone with at least two vertices is $T_1$
\cite[Proposition~4(ii)]{Mul95}, so the path with facets $\{1,2\},\{1,3\}$, and with it the
whole family $\AF^m\times X_\Gamma$, $m\geq1$, and every complex produced by
\cref{cor:stabilized}, is fully exposed but lies outside his theorem; conversely the boundary of
a triangle is $T_1$ yet fully hidden, so it lies outside ours. Neither theorem implies the other on
the overlap, since an abstract isomorphism of ind-groups carries no a priori compatibility at the
identity along which to differentiate it.}
\end{remark}
\rcomment{una locura lo largo dr la siguinte remark, yo me hubiera qudado solo con la
menci\'on en la introducci\'on. dejo una propuesta m\'as breve si se decide mantener (Yo
elimnar\'ia la remark).}

\begin{example}\label{ex:overlap}
The two hypotheses do overlap (\cref{fig:overlap}). The four shaded triangles are the facets $\{1,2,3\},\{1,4,5\},\{2,4,6\},\{3,5,6\}$ on $[6]$; any two of them meet in exactly one vertex. Removing any vertex from a facet leaves a pair of vertices contained in no other facet, so every facet is exposed and the complex is fully exposed; and each vertex is the intersection of the two facets through it, so the complex is also a $T_1$-complex in the sense of \cite{Mul95} (\cref{rem:muller-comparison}).
\begin{figure}[H]
\centering
\begin{tikzpicture}[scale=1.0, v/.style={circle,fill,inner sep=1.3pt}]
\fill[gray!16] (-0.605,0.215)--(0.605,0.215)--(0,-0.830)--cycle;
\fill[gray!16] (-0.605,0.215)--(0,2.4)--(-2.2,-1.4)--cycle;
\fill[gray!16] (0.605,0.215)--(0,2.4)--(2.2,-1.4)--cycle;
\fill[gray!16] (0,-0.830)--(-2.2,-1.4)--(2.2,-1.4)--cycle;
\draw (-0.605,0.215)--(0.605,0.215)--(0,-0.830)--cycle;
\draw (-0.605,0.215)--(0,2.4)--(-2.2,-1.4)--cycle;
\draw (0.605,0.215)--(0,2.4)--(2.2,-1.4)--cycle;
\draw (0,-0.830)--(-2.2,-1.4)--(2.2,-1.4)--cycle;
\node[v,label={[shift={(-0.10,0.02)}]left:{\small $1$}}] at (-0.605,0.215) {};
\node[v,label={[shift={(0.10,0.02)}]right:{\small $2$}}] at (0.605,0.215) {};
\node[v,label=below:{\small $3$}] at (0,-0.830) {};
\node[v,label=above:{\small $4$}] at (0,2.4) {};
\node[v,label=left:{\small $5$}] at (-2.2,-1.4) {};
\node[v,label=right:{\small $6$}] at (2.2,-1.4) {};
\end{tikzpicture}
\caption{A complex that is both fully exposed and $T_1$.}
\label{fig:overlap}
\end{figure}
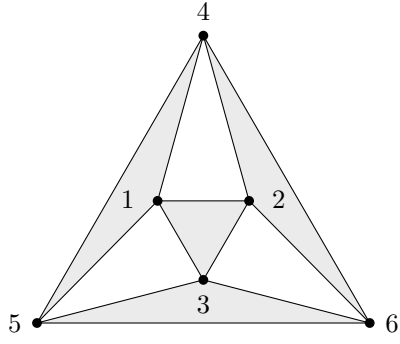
\end{example}

The hypothesis cannot be removed (\cref{prop:same-aut}), but it is satisfied after one
stabilization. We first record the baseline, which needs no automorphism groups at all.

\begin{corollary}\label{cor:sr-cancellation}
Let $\Delta$ and $\Delta'$ be simplicial complexes on $[n]$ and $[n']$. If
$X_\Delta\times\AF^1\cong X_{\Delta'}\times\AF^1$, then $\Delta\cong\Delta'$ and
$X_\Delta\cong X_{\Delta'}$.
\end{corollary}

\begin{proof}
Choose new vertices $v\notin[n]$ and $v'\notin[n']$. By \eqref{eq:cone-ring} the hypothesis reads
$\KK[\Delta\ast v]\cong\KK[\Delta'\ast v']$ as $\KK$-algebras, so $\Delta\ast v\cong\Delta'\ast v'$
by Bruns--Gubeladze \cite{BG96} and $\Delta\cong\Delta'$ by \cref{lem:cone-cancel}. Then
$\KK[\Delta]\cong\KK[\Delta']$, i.e.\ $X_\Delta\cong X_{\Delta'}$.
\end{proof}

The next statement is strictly stronger: isomorphic affine varieties have isomorphic automorphism ind-groups, so the hypothesis of \cref{cor:sr-cancellation} implies the one below.

\begin{corollary}\label{cor:stabilized}
Let $\Delta$ and $\Delta'$ be simplicial complexes on $[n]$ and $[n']$, not assumed fully exposed.
If
$$
\Aut\bigl(X_\Delta\times\AF^1\bigr)\ \cong\ \Aut\bigl(X_{\Delta'}\times\AF^1\bigr)
$$
as ind-groups, then $\Delta\cong\Delta'$ and $X_\Delta\cong X_{\Delta'}$.
\end{corollary}

\begin{proof}
Choose new vertices $v\notin[n]$ and $v'\notin[n']$. By \eqref{eq:cone-ring} we have
$X_{\Delta\ast v}\cong X_\Delta\times\AF^1$ and $X_{\Delta'\ast v'}\cong X_{\Delta'}\times\AF^1$,
and isomorphic affine varieties have isomorphic automorphism ind-groups; so the hypothesis reads
$\Aut(X_{\Delta\ast v})\cong\Aut(X_{\Delta'\ast v'})$. The apex $v$ lies in every facet of
$\Delta\ast v$, so $\Delta\ast v$ is fully exposed by \cref{lem:cone-roots}\ref{cr.3}, and likewise
for $\Delta'\ast v'$; also $\Delta\ast v$ satisfies the standing convention, by
\eqref{eq:cone-ring}. Hence \cref{thm:main} applies and gives $\Delta\ast v\cong\Delta'\ast v'$,
and \cref{lem:cone-cancel} gives $\Delta\cong\Delta'$, whence $\KK[\Delta]\cong\KK[\Delta']$ and
$X_\Delta\cong X_{\Delta'}$.
\end{proof}

The complexes behind the sharpness statement illustrate the point.

\begin{corollary}\label{cor:product-line}
Let $\Delta_1,\Delta_2$ be the complexes of \cref{ex:rigid-invisible}. Then the ind-groups
$\Aut(X_{\Delta_1}\times\AF^1)$ and $\Aut(X_{\Delta_2}\times\AF^1)$ are not isomorphic, although
$\Aut(X_{\Delta_1})\cong\Aut(X_{\Delta_2})$.
\end{corollary}

\begin{proof}
$\Delta_1\not\cong\Delta_2$, the multisets of facet cardinalities being $(2,2,2,2)$ and
$(2,2,2,1)$, so \cref{cor:stabilized} gives the first assertion; the second is
\cref{prop:same-aut}.
\end{proof}

In particular the automorphism ind-group of a stabilization is \emph{not} determined by the
automorphism ind-group: $\Aut(X_{\Delta_1})$ and $\Aut(X_{\Delta_2})$ are isomorphic while
$\Aut(X_{\Delta_1}\times\AF^1)$ and $\Aut(X_{\Delta_2}\times\AF^1)$ are not.

{\small
\bibliographystyle{alpha}
\bibliography{ref}

@article{Fil82,
  author = {R. P. Filipkiewicz},
  title = {Isomorphisms between diffeomorphism groups},
  journal = {Ergodic Theory Dynam. Systems},
  volume = {2},
  pages = {159--171},
  year = {1982}
}

@article{AG17,
  author = {I. Arzhantsev and S. Gaifullin},
  title = {The automorphism group of a rigid affine variety},
  journal = {Math. Nachr.},
  volume = {290},
  number = {5--6},
  pages = {662--671},
  year = {2017}
}

@article{BB66,
  author  = {A. Bia{\l}ynicki-Birula},
  title   = {Remarks on the action of an algebraic torus on $k^n$. {I}},
  journal = {Bull. Acad. Polon. Sci. S\'{e}r. Sci. Math. Astronom. Phys.},
  volume  = {14},
  pages   = {177--181},
  year    = {1966},
}

@article{AR15,
  author = {I. Arzhantsev and E. Romaskevich},
  title = {Additive actions on toric varieties},
  journal = {Proc. Amer. Math. Soc.},
  volume = {145},
  pages = {1865--1879},
  year = {2017}
}

@book{BG09,
  author = {W. Bruns and J. Gubeladze},
  title = {Polytopes, Rings, and {K}-Theory},
  series = {Springer Monographs in Mathematics},
  publisher = {Springer},
  address = {Dordrecht},
  year = {2009}
}

@article{BG96,
  author = {W. Bruns and J. Gubeladze},
  title = {Combinatorial invariance of {S}tanley--{R}eisner rings},
  journal = {Georgian Math. J.},
  volume = {3},
  number = {4},
  pages = {315--318},
  year = {1996}
}

@incollection{BG02,
  author = {W. Bruns and J. Gubeladze},
  title = {Polyhedral algebras, arrangements of toric varieties, and their groups},
  booktitle = {Computational Commutative Algebra and Combinatorics (Osaka, 1999)},
  series = {Adv. Stud. Pure Math.},
  volume = {33},
  publisher = {Math. Soc. Japan},
  address = {Tokyo},
  pages = {1--51},
  year = {2002}
}

@article{BS95,
  author = {P. Brumatti and A. Simis},
  title = {The module of derivations of a {S}tanley--{R}eisner ring},
  journal = {Proc. Amer. Math. Soc.},
  volume = {123},
  pages = {1309--1318},
  year = {1995}
}

@misc{CERUvS25,
  author = {N. Chen and L. Esser and A. Regeta and C. Urech and I. van Santen},
  title = {Characterizing varieties using birational transformations},
  year = {2025},
  note = {\href{https://arxiv.org/abs/2512.02837}{arXiv:2512.02837}}
}

@article{CRX23,
  author = {S. Cantat and A. Regeta and J. Xie},
  title = {Families of commuting automorphisms, and a characterization of the affine space},
  journal = {Amer. J. Math.},
  volume = {145},
  number = {2},
  pages = {413--434},
  year = {2023}
}

@article{Dem70,
  author = {M. Demazure},
  title = {Sous-groupes alg\'ebriques de rang maximum du groupe de {C}remona},
  journal = {Ann. Sci. \'Ecole Norm. Sup. (4)},
  volume = {3},
  pages = {507--588},
  year = {1970}
}

@article{DL24,
  author = {R. D{\'\i}az and A. Liendo},
  title = {On the automorphism group of non-necessarily normal affine toric varieties},
  journal = {Int. Math. Res. Not. IMRN},
  volume = {2024},
  number = {2},
  pages = {1624--1649},
  year = {2024}
}

@article{DLA24,
  author = {R. D{\'\i}az and G. Lucchini Arteche},
  title = {Finite-dimensional monomial algebras are determined by their automorphism group},
  journal = {Transform. Groups},
  year = {2026},
  pages = {1--14},
  note = {Published online, 25 May 2026},
  doi = {10.1007/s00031-026-09970-2}
}

@misc{DLMFR24,
  author = {R. D{\'\i}az and A. Liendo and G. Manzano-Flores and A. Regeta},
  title = {The automorphism groups of zero-dimensional monomial algebras},
  year = {2024},
  note = {\href{https://arxiv.org/abs/2408.02197}{arXiv:2408.02197}}
}

@article{DLR26,
  author = {R. D{\'\i}az and A. Liendo and A. Regeta},
  title = {On the characterization of affine toric varieties by their automorphism group},
  journal = {Israel J. Math.},
  year = {2026},
  pages = {1--23},
  note = {Published online, 19 February 2026},
  doi = {10.1007/s11856-026-2900-0}
}

@misc{FK18,
  author = {J.-P. Furter and H. Kraft},
  title = {On the geometry of the automorphism groups of affine varieties},
  year = {2018},
  note = {\href{https://arxiv.org/abs/1809.04175}{arXiv:1809.04175}}
}

@article{GPT14,
  author = {P. D. Gonz{\'a}lez P{\'e}rez and B. Teissier},
  title = {Toric geometry and the {S}emple--{N}ash modification},
  journal = {Rev. R. Acad. Cienc. Exactas F{\'\i}s. Nat. Ser. A Mat. RACSAM},
  volume = {108},
  number = {1},
  pages = {1--48},
  year = {2014}
}

@book{HH11,
  author = {J. Herzog and T. Hibi},
  title = {Monomial Ideals},
  series = {Graduate Texts in Mathematics},
  volume = {260},
  publisher = {Springer-Verlag London},
  year = {2011}
}

@article{Kra17,
  author = {H. Kraft},
  title = {Automorphism groups of affine varieties and a characterization of affine $n$-space},
  journal = {Trans. Moscow Math. Soc.},
  volume = {78},
  pages = {171--186},
  year = {2017}
}

@article{KRvS21,
  author = {H. Kraft and A. Regeta and I. van Santen},
  title = {Is the affine space determined by its automorphism group?},
  journal = {Int. Math. Res. Not. IMRN},
  volume = {2021},
  number = {6},
  pages = {4280--4300},
  year = {2021}
}

@article{Lie10,
  author = {A. Liendo},
  title = {Affine $\mathbb{T}$-varieties of complexity one and locally nilpotent derivations},
  journal = {Transform. Groups},
  volume = {15},
  number = {2},
  pages = {389--425},
  year = {2010}
}

@article{LRU23,
  author = {A. Liendo and A. Regeta and C. Urech},
  title = {Characterization of affine surfaces with a torus action by their automorphism groups},
  journal = {Ann. Sc. Norm. Super. Pisa Cl. Sci. (5)},
  volume = {24},
  number = {1},
  pages = {249--289},
  year = {2023}
}

@book{MS05,
  author = {E. Miller and B. Sturmfels},
  title = {Combinatorial Commutative Algebra},
  series = {Graduate Texts in Mathematics},
  volume = {227},
  publisher = {Springer-Verlag},
  address = {New York},
  year = {2005}
}

@article{Mul95,
  author = {G. M{\"u}ller},
  title = {Lie algebras attached to simplicial complexes},
  journal = {J. Algebra},
  volume = {177},
  number = {1},
  pages = {132--141},
  year = {1995}
}

@article{RUvS25,
  author = {A. Regeta and C. Urech and I. van Santen},
  title = {Group theoretical characterizations of rationality},
  journal = {Invent. Math.},
  volume = {241},
  pages = {681--716},
  year = {2025}
}

@article{RvS21,
  author = {A. Regeta and I. van Santen},
  title = {Characterizing smooth affine spherical varieties via the automorphism group},
  journal = {J. \'Ec. polytech. Math.},
  volume = {8},
  pages = {379--414},
  year = {2021}
}

@book{Eis95,
  author    = {David Eisenbud},
  title     = {Commutative Algebra with a View Toward Algebraic Geometry},
  series    = {Graduate Texts in Mathematics},
  volume    = {150},
  publisher = {Springer-Verlag},
  address   = {New York},
  year      = {1995},
}

@article{RvS25,
  author = {A. Regeta and I. van Santen},
  title = {Maximal commutative unipotent subgroups and a characterization of affine spherical varieties},
  journal = {J. Eur. Math. Soc.},
  year = {2025},
  doi = {10.4171/JEMS/1651},
  note = {Published online first}
}

@article{Tad09,
  author = {Y. Tadesse},
  title = {Derivations preserving a monomial ideal},
  journal = {Proc. Amer. Math. Soc.},
  volume = {137},
  number = {9},
  pages = {2935--2942},
  year = {2009}
}

@book{Bor91,
  author = {A. Borel},
  title = {Linear Algebraic Groups},
  series = {Graduate Texts in Mathematics},
  volume = {126},
  edition = {Second},
  publisher = {Springer-Verlag},
  address = {New York},
  year = {1991}
}

@book{Spr98,
  author    = {T. A. Springer},
  title     = {Linear Algebraic Groups},
  series    = {Progress in Mathematics},
  volume    = {9},
  edition   = {Second},
  publisher = {Birkh\"auser},
  address   = {Boston},
  year      = {1998}
}

@article{AC04,
  author  = {Altmann, Klaus and Christophersen, Jan Arthur},
  title   = {Cotangent cohomology of {Stanley--Reisner} rings},
  journal = {Manuscripta Math.},
  volume  = {115},
  number  = {3},
  pages   = {361--378},
  year    = {2004},
  doi     = {10.1007/s00229-004-0496-3}
}

@article{AFKKZ13,
  author  = {Arzhantsev, Ivan and Flenner, Hubert and Kaliman, Shulim and
             Kutzschebauch, Frank and Zaidenberg, Mikhail},
  title   = {Flexible varieties and automorphism groups},
  journal = {Duke Math. J.},
  volume  = {162},
  number  = {4},
  pages   = {767--823},
  year    = {2013},
  doi     = {10.1215/00127094-2080132}
}

@article{For98,
  author  = {Forman, Robin},
  title   = {Morse theory for cell complexes},
  journal = {Adv. Math.},
  volume  = {134},
  number  = {1},
  pages   = {90--145},
  year    = {1998}
}

@book{Fre17,
  author    = {Freudenburg, Gene},
  title     = {Algebraic Theory of Locally Nilpotent Derivations},
  edition   = {Second},
  series    = {Encyclopaedia of Mathematical Sciences},
  volume    = {136},
  publisher = {Springer},
  address   = {Berlin},
  year      = {2017}
}

@book{Koz08,
  author    = {Kozlov, Dmitry},
  title     = {Combinatorial Algebraic Topology},
  series    = {Algorithms and Computation in Mathematics},
  volume    = {21},
  publisher = {Springer},
  address   = {Berlin},
  year      = {2008}
}

@article{Vil90,
  author  = {Villarreal, Rafael H.},
  title   = {Cohen--Macaulay graphs},
  journal = {Math. Ann.},
  volume  = {66},
  number  = {3},
  pages   = {277--293},
  year    = {1990}
}

@article{Kal87,
  author  = {Kalai, Gil},
  title   = {Rigidity and the lower bound theorem {I}},
  journal = {Invent. Math.},
  volume  = {88},
  number  = {1},
  pages   = {125--151},
  year    = {1987},
  doi     = {10.1007/BF01405094}
}

@article{AC10,
  author  = {Altmann, Klaus and Christophersen, Jan Arthur},
  title   = {Deforming {Stanley--Reisner} schemes},
  journal = {Math. Ann.},
  volume  = {348},
  number  = {3},
  pages   = {513--537},
  year    = {2010},
  doi     = {10.1007/s00208-010-0490-x}
}

@article{FrH08,
  author  = {Francisco, Christopher A. and H{\`a}, Huy T{\`a}i},
  title   = {Whiskers and sequentially {Cohen--Macaulay} graphs},
  journal = {J. Combin. Theory Ser. A},
  volume  = {348},
  number  = {2},
  pages   = {304--316},
  year    = {2008},
}

@article{BGa25,
  author  = {Borovik, Viktoriia and Gaifullin, Sergey},
  title   = {Isolated torus invariants and automorphism groups of rigid varieties},
  journal = {J. Algebra},
  volume  = {666},
  pages   = {821--839},
  year    = {2025},
  doi     = {10.1016/j.jalgebra.2024.12.009}
}
}

\end{document}